\documentclass[a4paper,11pt,reqno]{amsart}
\usepackage[margin=1.2in]{geometry}

\usepackage[T1]{fontenc}
\usepackage[utf8]{inputenc}
\usepackage{amsfonts}
\usepackage{amssymb}
\usepackage{amsthm}
\usepackage{amsmath}
\usepackage{bm}
\usepackage{mathtools}
\usepackage{physics}

\usepackage{thmtools}
\usepackage{thm-restate}

\usepackage{hyperref}
\usepackage{enumerate}
\usepackage{enumitem}
\usepackage{eurosym}
\usepackage{graphicx}
\usepackage{colortbl}
\usepackage{epstopdf}
\usepackage[usenames,dvipsnames]{xcolor}

\newcommand*{\myproofname}{Proof}

\newcommand*{\R}{\mathbb{R}}

\newcommand*{\prob}{\mathbb{P}}
\newcommand*{\freen}{\mathcal{F}}

\renewcommand{\norm}[1]{\lVert#1\rVert}

\newtheorem{theorem}{Theorem}
\newtheorem{lemma}[theorem]{Lemma}

\newtheorem{corollary}[theorem]{Corollary}

\theoremstyle{definition}
\newtheorem{definition}{Definition}
\newtheorem{remark}{Remark}

\DeclareMathOperator{\diag}{diag}
\DeclareMathOperator{\dist}{dist}

\title[WSM and SSM for the anti-ferromagnetic Potts model on trees]{Near optimal bounds for weak and strong spatial mixing for the anti-ferromagnetic Potts model on trees}

\date{\today}
\thanks{This research was supported by the Netherlands Organisation of Scientific Research (NWO): VI.Vidi.193.068.}

\author{Ferenc Bencs}
\email{\texttt{ferenc.bencs@gmail.com}}
\author{Khallil Berrekkal}
\email{\texttt{k.berrekkal@uva.nl}}
\author{Guus Regts}
\email{\texttt{guusregts@gmail.com}}
\address{Korteweg de Vries Institute for Mathematics, University of Amsterdam, P.O. Box 94248 
1090 GE Amsterdam, The Netherlands.}
\begin{document}
\maketitle

\begin{abstract}
We show that the anti-ferromagnetic Potts model on trees exhibits strong spatial mixing for a near-optimal range of parameters.
Our work complements recent results of Chen, Liu, Mani, and Moitra~\cite{chen2023strong} who showed this to be true in the infinite temperature setting, corresponding to uniform proper colorings. 
We furthermore prove weak spatial mixing results complementing results in~\cite{chen2023strong}.
\\
\quad \\
\noindent
{\bf Keywords:} Potts model, weak spatial mixing, strong spatial mixing, anti-ferromagnetic, infinite regular tree
\end{abstract}
	
\section{Introduction}
Consider a uniformly random proper $q$-coloring of the vertices of the (infinite) rooted $d$-ary tree conditioned on a given proper coloring at distance $h$ from the root. 
Is the marginal distribution of the root vertex close to the uniform distribution?
The answer to this question turns out to depend on the relation between $q$ and $d$: if $q\leq d+1$, then the answer could be `no' depending on the given coloring, while if $q>d+1$, the answer is `yes' for any given coloring, as was shown about twenty years ago by Jonasson~\cite{Jonasson2002}. 
This property is often referred to as \emph{weak spatial mixing} in computer science; it implies uniqueness of the Gibbs measure on the infinite $(d+1)$-regular tree~\cite{BW2002}. 
The stronger property known as \emph{strong spatial mixing}, which requires that the marginal distributions of the root vertex for any two arbitrary partial colorings of the tree are close in terms of the distance to the nearest disagreement of the given partial colorings (see below for a precise definition), turned out to be more difficult to establish. 
In a very recent breakthrough it was shown by Chen, Liu, Mani, and Moitra~\cite{chen2023strong}, that strong spatial mixing holds, provided $q\geq d+4$, thereby significantly improving on earlier results~\cite{GamarniketalSSM,GalanisetalimprovedSSM} where $q$ was required to be larger than $1.58d$ to conclude strong spatial mixing.

In the present paper, we consider the anti-ferromagnetic Potts model, in which for a parameter $w\in [0,1]$ a $q$-coloring (not necessarily proper) of the vertices of a finite graph is selected proportionally to $w$ raised to the number of monochromatic edges in the coloring (setting $w=0$ corresponds to the uniform distribution on proper colorings). 
It is a folklore conjecture in statistical physics that for $w>0$ the model exhibits weak spatial mixing if and only if $w\geq 1-\tfrac{q}{d+1}$.
The lower bound on $w$ is known to be tight~\cite{Peruggietalgeneral,Peruggietaldiagrams,gal15}.
%(where one should relax the definition of exponential decay to just some decay when $w=1-\tfrac{q}{d+1}$ in the definition of weak spatial mixing).
For $q=2$ the conjecture has been known to be true for some time~\cite{Georgiibook,SST2-spinantiferro}, and in fact, the model even exhibits strong spatial mixing when $w>1-\tfrac{q}{d+1}$~\cite{SST2-spinantiferro}.
Confirming the conjecture for $q\geq 3$ turned out to be much more difficult, and it was not until a few years ago that this was done for $q=3$~\cite{GalanisetalUniquenessq=3} and $q=4$ and $d\geq 5$~\cite{deBoeretalq=4}.
An asymptotic form of the conjecture was recently shown to be true in the sense that for each $q\geq 3$, there exists $d_0$, such that for all $d\geq d_0$ the conjecture is true~\cite{Bencsetalallq}.
Unfortunately, the exact dependence of $d_0$ on $q$ is not known.
Chen, Liu, Mani, and Moitra~\cite{chen2023strong} recently proved that the model has weak spatial mixing for a range of parameters that approaches the conjectured threshold as both $q,d \to \infty$.

In the present paper, we complement the results from~\cite{Bencsetalallq} and~\cite{chen2023strong} by proving near optimal weak and strong spatial mixing with concrete bounds on the parameter $w$ in terms of $q$ and $d$.
Before we state our main results, we first give a precise definition of weak and strong spatial mixing.

\subsubsection*{Definitions}
Let $q > 0$ be an integer and let $T = (V, E)$ be a finite tree with the root vertex denoted as $v$. 
We introduce the term \textit{partially $q$-colored tree} to describe a triple $(T, \Lambda, \tau)$, where $\Lambda \subset V$ and $\tau: \Lambda \to [q] = \{1, \ldots, q\}$. 
	We often refer to $\tau$ as a \textit{partial coloring} or \emph{boundary condition} of the tree, and any vertex not contained in $\Lambda$ is referred to as a \textit{free} vertex. 
	%(We will assume that the distance from the root vertex $v$ to $\Lambda$ is at least 3. )
	
	For a given coloring $\psi: V \to [q]$, we define $m(\psi)$ as the number of monochromatic edges, i.e. the number of edges whose endpoints share the same color. 
	The partition function, $Z_T(w)$, of the $q$-state Potts model of $T$ with boundary condition $\tau$ is expressed as
	\begin{align}\label{eq:def pf}
		Z_{T}(w) \coloneqq \sum_{\substack{\psi: V \to [q]  \\ \psi|_\Lambda = \tau } } w^{m(\psi) }.
	\end{align}
 
	For $w \geq 0$, there is an associated probability measure $\mathbb{P}_{T;w}$ on the space of all colorings respecting the boundary conditions, with the probability mass function defined as:
	
	\begin{align*}
		\mu_{T;w}(\psi) = \dfrac{w^{m(\psi)}}{Z_T(w)}.
	\end{align*}
	
	To denote random variables, we use capital letters. 
	In particular, $\mathbb{P}_{T;w}[\Phi(v) = j\mid \tau]$ denotes the probability of assigning color $j$ to vertex $v$ when sampling a coloring from this distribution. 
	In cases where $w$ is evident from the context, we use the shorthand notation $\mathbb{P}_{T}$ instead of $\mathbb{P}_{T;w}$. 
 
	\begin{definition}[Weak Spatial Mixing (WSM)]
 Let $\mathcal{T}$ be a collection of rooted trees.
				The $q$-state Potts model on $\mathcal{T}$ at parameter $w\geq 0$ exhibits \emph{weak spatial mixing} (\emph{WSM}) with exponential decay rate of $r\in (0,1)$, if there exists a constant $C>0$, such that for any finite rooted tree $(T,v)\in \mathcal{T}$, any $\Lambda \subset V(T) \setminus \{v\}$, and any boundary condition $\tau: \Lambda \to [q]$, as well as any color $i \in [q]$, it holds that
		\begin{align*}
		\left|\mathbb{P}_{T;w}[\Phi(v) = i|\tau]-\tfrac{1}{q}\right|\leq Cr^{\dist(v,\Lambda)},
		\end{align*}
  where $\dist(v,\Lambda)$ denotes the graph distance from the root vertex $v$ to the set $\Lambda.$
			\end{definition}

 When considering two distinct boundary conditions, $\tau$ and $\tau'$, defined on the same vertex set $\Lambda$, we can compare the marginal probabilities of the root vertex $v$ receiving color $i$ for both boundary conditions. 
	If the difference in marginal probabilities tends to zero as the distance increases,  we speak of strong spatial mixing.
	
	\begin{definition}[Strong Spatial Mixing (SSM)]
  Let $\mathcal{T}$ be a collection of rooted trees.
				The $q$-state Potts model on $\mathcal{T}$ at parameter $w\geq 0$ exhibits \emph{strong spatial mixing} (\emph{SSM}) with exponential decay rate of $r\in (0,1)$, if there exists a constant $C>0$, such that for any  finite rooted tree $(T,v)\in \mathcal{T}$, any $\Lambda \subset V(T) \setminus \{v\}$, and any two boundary conditions $\tau, \tau': \Lambda \to [q]$ differing on $\Delta_{\tau,\tau'} \coloneqq \{ u \in \Lambda \mid \tau(u) \neq \tau'(u) \}  \subset V(T)$, as well as any color $i \in [q]$, it holds that
		\begin{align*}
			\big| \mathbb{P}_{T;w}[\Phi(v) = i|\tau] -  \mathbb{P}_{T;w}[\Phi(v) = i|\tau']\big| \leq Cr^{\dist(v, \Delta_{\tau, \tau'})}.
		\end{align*}
			\end{definition}

The following lemma implies that SSM is indeed a stronger property than WSM.
\begin{lemma}\label{lem:SSM->WSM}
Let $q>0$ be an integer.
The $q$-state Potts model at parameter $w\geq 0$ exhibits weak spatial mixing with exponential decay rate $r\in (0,1)$ on a family of rooted trees $\mathcal{T}$, if there exists a constant $C>0$ such that for each rooted tree $(T,v)\in \mathcal{T}$, and any two boundary conditions $\tau,\tau'$ on $\Lambda$, where $\Lambda=\{u\in V(T)~|~\dist(u,v)=t\}$ for some $t\ge 1$, and any color $i\in [q]$,
\begin{align*}
			\big| \mathbb{P}_{T;w}[\Phi(v) = i|\tau] -  \mathbb{P}_{T;w}[\Phi(v) = i|\tau']\big| \leq Cr^{t}.
		\end{align*}
\end{lemma}
\begin{proof}
Let $(T,v)\in \mathcal{T}$ and let $\Lambda=\{u\in V(T)~|~\dist(u,v)=t\}$ for some $t \ge 1$.
Let $\tau_{\text{max}}$ (resp. $\tau_{\text{min}}$) be a boundary condition on $\Lambda$ that maximizes (resp. minimizes) $\prob_{T;w}[\Phi(v)=i\mid \tau]$ for a given color $i$.
Then by symmetry 
\[
\prob_{T;w}[\Phi(v)=i\mid \tau_{\text{min}}]\leq \tfrac{1}{q}\leq \prob_{T;w}[\Phi(v)=i\mid \tau_{\text{max}}].
\]
By assumption, we have 
\[
\prob_{T;w}[\Phi(v)=i\mid \tau_{\text{max}}]-\prob_{T;w}[\Phi(v)=i\mid \tau_{\text{min}}]\leq C r^{\dist(v,\Lambda)}.
\] 
 Therefore, $ \prob_{T;w}[\Phi(v)=i\mid \tau_{\text{max}}]-1/q$ is equal to
\begin{align*}
  \prob_{T;w}[\Phi(v)=i\mid \tau_{\text{max}}]-\prob_{T;w}[\Phi(v)=i\mid \tau_{\text{min}}]+\prob_{T;w}[\Phi(v)=i\mid \tau_{\text{min}}]-1/q
\end{align*}
and hence is bounded from above by $Cr^{\dist(v,\Lambda)}$.
It follows in a similar way that
\[
\prob_{T;w}[\Phi(v)=i\mid \tau_{\text{min}}]-1/q\geq -Cr^{\dist(v,\Lambda)},
\]
implying that the model exhibits WSM with the desired decay rate.
\end{proof}

\subsubsection*{Main results}
We denote by $\mathcal{T}_\Delta$ the collection of all rooted trees of maximum degree at most $\Delta$, where the root has degree at most $\Delta-1$.
Our first main theorem concerns SSM and complements the SSM result of~\cite{chen2023strong} for the case $w=0$.
    \begin{theorem}\label{thm:main ssm}
        There exists a constant $K>0$, such that for any $q\geq 3$, and any $d$ such that $d+1\ge \tfrac{e-1/2}{e-1}q$, the $q$-state Potts model at parameter $w$ on $\mathcal{T}_{d+1}$ exhibits SSM, provided
        \[
        1> w\geq 1-\frac{q}{d+1}\left(1-\frac{K}{d+1}\right).
        \]
    \end{theorem}
In particular, this result shows that as $d\to \infty$, the bound on $w$ approaches the optimal bound of $1-\tfrac{q}{d+1}$ at a rate that is linear in $1/d$.
\begin{remark}\label{rem:SSM 1}
The constant $K$ in the theorem above in fact depends on $d$ and converges to $\tfrac{5}{2}e^2$ as $d\to \infty$, as one can verify by inspecting our proof given below.
We note that when the ratio between $d+1$ and $q$ is less than $(e-1/2)/(e-1)$, such a constant $K$ still exists.
However, it is unclear whether this constant is uniformly bounded. 
See also Remark~\ref{rem:SSM 2}.
\end{remark}
\begin{remark}
This result can easily be extended to the slightly larger class of rooted trees where we allow the root vertex to have degree $\Delta$. 
This follows because the marginal probability of the root vertex depends continuously on those of its neighbors, and if these are close to each other for two boundary conditions, the same will be true for the marginal probabilities at the root vertex.      
\end{remark}

Our second main result concerns WSM, for which we can get a better range of parameters than in Theorem~\ref{thm:main ssm}.

    \begin{theorem}\label{thm:main wsm}
        For each integer $q\geq 3$ there exists a constant $K_q''=K_q''(d)>0$ such that for any $d\ge q+2$, the $q$-state Potts model at parameter $w$ on $\mathcal{T}_{d+1}$ exhibits WSM provided
        \[
        1> w\geq 1-\frac{q}{d+1}\left(1-\frac{K_q''}{d+1}\right).
        \]
        Moreover, $\limsup_{d\to\infty}K''_q(d)=5e/2$ and  $K''_q$ is uniformly bounded  in $d$ and $q$. %\feri{Or say $\sup_{q\ge 3}\max_{d\ge q+2} K''_q(d)<40$} 
    \end{theorem}

This result yields very concrete bounds on the parameter $w$ for the model to exhibit WSM, thereby nicely complementing the exact result that holds for large values of $d$ from~\cite{Bencsetalallq}. 
In~\cite{chen2023strong} a similar result was established. However, in that work, the factor in front of $\tfrac{q}{d+1}$ is of the form $1-\tfrac{C}{d+1}-\frac{C'}{q}$, for positive constants $C,C'$ and thus is asymptotically optimal only when \emph{both} $q$ and $d$ tend to infinity.

We refer to~\cite{gu2022power} for the optimal possible rate in Theorem~\ref{thm:main wsm}. 

\subsection{Motivation from computer science}    
As mentioned earlier, proving WSM/SSM is motivated from the perspective of statistical physics as it implies uniqueness of the Gibbs measure~\cite{BW2002}.
It is additionally motivated from a computer science perspective. 
In~\cite{Blancasamplin}, the authors showed that for parameters $q,d,w$, such that the $q$-state Potts model at parameter $w$ exhibits WSM on trees of maximum degree at most $d$, there is an efficient (randomized) algorithm to sample from the Potts distribution on random $d$-regular graphs.
Theorem~\ref{thm:main wsm} thus gives a wide range of parameters for which this efficient sampling is possible.

In~\cite{chen2023strong} the authors used their proof for SSM of the Potts model at $w=0$ (corresponding to proper colorings) to also show the rapid mixing of the Glauber dynamics on large girth graphs. 
This was done by building upon the recently established and powerful spectral independence approach to analyze Glauber dynamics~\cite{anarispectral,chenoptimal,Blancaetalmixing}.
While the proof given in~\cite{chen2023strong} is specifically tailored to the setting of the $w=0$ case, we suspect that our proof for SSM for $w>0$ also can be adapted to demonstrate rapid mixing of the Glauber dynamics for large girth graphs. 
This is because our proof of Theorem~\ref{thm:main ssm} has a similar flavor as the proof approach in~\cite{chen2023strong} (we say more about this in the following subsection).
We choose not to pursue this here so as to keep the paper concise.

\subsection{Approach and organization}
The basic idea is to express the marginal probabilities of the root vertex as a function $F$ of the marginals of its neighbors. 
(In statistical physics this is sometimes referred to as the cavity method.)
The next step is to analyze this function and show that it contracts somehow, provided that $w\geq 1-\tfrac{q}{d+1}$ (when $w<1-\frac{q}{d+1}$ it is not difficult to see that it cannot be contractive).
In~\cite{Jonasson2002,GalanisetalUniquenessq=3,deBoeretalq=4,Bencsetalallq} this is done  by looking at two iterations of $F$ and showing that this contracting somehow. Having established this contracting behavior it then follows with a simple inductive proof that the model exhibits WSM.
Unfortunately, this two-step iteration is not straightforward to analyze.
In~\cite{Bencsetalallq} this is done by considering a suitable change of coordinates coupled with a geometric approach, which allows to show WSM for all $w\geq 1-\tfrac{q}{d+1}$ (provided $d$ is large enough).

Looking at two iterations of $F$ means taking the information of the depth $2$ neighborhood of the root vertex into account. 
In case this contains vertices whose color is fixed by the boundary condition (called \emph{fixed vertices}) this complicates matters significantly, which can not be avoided in case one wants to show SSM (for WSM one can simply assume that there are no fixed vertices at distance at most two from the root).
The approach taken in~\cite{chen2023strong} gets around this in a nice way by conjugating the function $F$ with a suitable chosen function thereby only relying on the first neighborhood (which still can contain vertices that are fixed, but in the setting of $w=0$ this turns out not to be an issue).
They show that the resulting function is contracting in one step with respect to the local modification of the squared $2$-norm, by analyzing the Jacobian of the conjugate of the function $F$, which takes a particularly nice form. 

Our approach is inspired both by~\cite{Bencsetalallq} and~\cite{chen2023strong}. 
Following~\cite{Bencsetalallq}, we do not work directly with the marginal probabilities, but rather with a certain ratio of partition functions that can easily be related to the marginal probabilities. 
Inspired by~\cite{chen2023strong} we conjugate this with a very simple function, the square root, obtaining \emph{square root ratios}.
The resulting function that expresses these square root ratios at the root vertex in terms of those at its neighbors has a convenient form, making it relatively easy to analyze its Jacobian of the function, which takes a similar form as in~\cite{chen2023strong}. 

Even though we work with different coordinates, our approach has strong parallels with the approach in~\cite{chen2023strong}. In particular we borrow some of their tools or variations of these.
A notable difference is that our choice of coordinates allows us to handle fixed vertices at distance one from the root in case $w>0$ and thereby allows us to prove SSM for a large range of parameters for the anti-ferromagnetic Potts model.

It is unclear to us how to approach the remaining regime for SSM, i.e. the regime $d+1\leq \tfrac{e-1/2}{e-1}q$. 
It is likely that a different change of coordinates will provide more insight.

In Section~\ref{sec:preliminaries}, we give a detailed outline of our approach, state the two main ingredients, and use these to provide inductive proofs of our main results.
In Sections~\ref{sec:contraction} and~\ref{sec: local weight}, we prove these two main ingredients. 
We conclude with some remarks and questions in Section~\ref{sec:conclusion}.

\section{Setup and detailed approach} \label{sec:preliminaries}
First, we will present a sufficient (and actually an equivalent) formulation of WSM and SSM using ratios of partition functions. 
Most of this follows from routine manipulations of the partition functions much like in~\cite{GalanisetalUniquenessq=3,deBoeretalq=4}, but we include the details for completeness.
Then, we will introduce the notations and the main propositions required to prove WSM and SSM through an inductive approach.

For a $w \in \mathbb{R}_{>0}$ and a free vertex $v$  of the partially $q$-colored tree $(T,\Lambda,\tau)$, we consider the ratio
\begin{align*}
    \widetilde{R}_{T,v;i}(w) \coloneqq \frac{Z_{T,v}^i (w)}{Z_{T - v}(w)},
\end{align*}
where $Z_{T,v}^i (w)$ denotes the partition function restricted to the colorings (respecting $\tau$) that assign color $i$ to vertex $v$. 
If $v$ is not free, say it is fixed to color $i$, then $\widetilde{R}_{T,v}(w)=e_i\in\mathbb{R}^{q}$, where $e_i$ is the $i-$th standard basis vector.
In case we want to highlight the boundary condition, we write  $\widetilde{R}^\tau_{T,v;i}(w)$ instead of $\widetilde{R}_{T,v;i}(w)$.

Let us make two important observations. 
First of all, we have
	\begin{align}\label{eq: observation1}
		\frac{\widetilde{R}_{T,v;i}(w)}{\sum_{j \in [q]}\widetilde{R}_{T,v;j}(w) } =  \mathbb{P}_{T;w}[\Phi(v) = i].
	\end{align}
On the other hand, for each coordinate of $\tilde R_{T,v}$, we have the following description. 
Let $v_1, \ldots, v_d$ be the children of $v$. 
We define $T_{v_j}$ as the connected component containing $v_j$ of the graph $T-v$. 
If we fix a color $i$, we define $\hat{T}_i$ as the forest obtained from $T- v$ by adding a leaf, $\hat{v}_j$, fixed to color $i$, to each vertex $v_j$. 
We denote by $\hat{T}_{i,\hat{v}_j}$ the connected component containing $\hat{v}_{j}$ in $\hat{T}_i$.
We drop the subscript $i$, to denote the same connected component, but for which $\hat{v}_j$ is a free vertex.
Again, we silently carry over the boundary condition to these trees.

Since $T$ is a tree, we have that $Z_{T,v}^i(w) = \prod_{j=1}^d Z_{\hat{T}_{i,\hat{v}_j}}(w) $. 
It follows that
\begin{align*}
    \Tilde{R}_{T,v;i}(w) = \frac{ \prod_{j=1}^d Z_{\hat{T}_{i,\hat{v}_j}}(w) }{Z_{T - v}(w)} =  \frac{ \prod_{j=1}^d Z^i_{\hat{T}_{\hat{v}_j}, \hat{v}_j}(w) }{ \prod_{j=1}^d Z_{\hat{T}_{\hat{v}_j} - \hat{v}_j}(w)}  = \prod_{j=1}^d  \Tilde{R}_{\hat{T}_{\hat{v}_j},\hat{v}_j;i}(w). 
\end{align*}
Writing out the ratio vectors in the trees $\hat{T}_{\hat{v}_j}$ and noting that $\hat{T}_{\hat{v}_j}-\hat{v_j}=T_{v_j}$, gives us
\begin{align*}
    \Tilde{R}_{\hat{T}_{\hat{v}_j},\hat{v}_j;i}(w) &= \frac{Z_{\hat{T}_{\hat{v}_j},\hat{v}_j}^i(w)}{Z_{\hat{T}_{\hat{v}_j} - \hat{v}_j}(w)} \\
    &= \frac{  \sum_{k \in [q] \setminus{\{i\}}} Z_{{T}_{{v}_j},v_j}^k(w) +  w Z_{{T}_{{v}_j},v_j}^i(w)}{ \sum_{k \in [q]} Z_{{T}_{{v}_j},v_j}^k(w)}\\
    &= 1+(w-1)\frac{Z_{{T}_{{v}_j},v_j}^i(w)}{ \sum_{k \in [q]} Z_{{T}_{{v}_j},v_j}^k(w)}\\
    &= 1+(w-1)\prob_{T_{v_j};w}[\Phi(v_j)=i],
\end{align*}
where $\Phi$ is a random sample from the Potts model on the tree $T_{v_j}$.
Also note that by dividing the numerator and denominator by $Z_{{T}_{{v}_j}-v_j}(w)$ in the penultimate equation, we get
  \begin{equation}\label{eq: observation3}
  \Tilde{R}_{\hat{T}_{\hat{v}_j},\hat{v}_j;i}(w)=1+(w-1)\frac{ \Tilde{R}_{T_{v_j},v_j ; i } (w)}{ \sum_{k \in [q]}  \Tilde{R}_{T_{v_j},v_j ; k }(w)}.
  \end{equation}
Thus
\begin{equation}\label{eq: observation2}
    \Tilde{R}_{T,v;i}(w)=\prod_{j=1}^{d}\left(1-(1-w)\prob_{T_{v_j};w}[\Phi(v_j)=i]\right).
\end{equation}

% Let us denote for simplicity $\tilde R$ and $\tilde R'$ the ratio vectors arising from coloring $T$ with $\tau$ and $\tau'$.

To prove SSM and WSM we will work with the ratio vectors defined as above. 
In what follows $\|\cdot \|$ always denotes the $2$-norm on $\mathbb{R}^q$: $\|x\|=\sqrt{\sum_{i=1}^q x_i^2}$ for $x\in \mathbb{R}^q$.

\begin{lemma}\label{lemma: ratio to SSM}
Let $d,q\in\mathbb{N}_{\geq 2}.$
If $w>0$, then the following condition implies that the Potts model at $w$ exhibits SSM with decay rate $r\in (0,1)$ on $\mathcal{T}_{d+1}$. 
There exists a constant $C'>0$ such that for any $(T,v)\in \mathcal{T}_{d+1}$ and any $\Lambda \subset V(T) \setminus \{v\}$, and any two boundary conditions $\tau, \tau': \Lambda \to [q]$ differing on $\Delta_{\tau,\tau'} = \{ u \in V(T) \mid \tau(u) \neq \tau'(u) \}  \subset V(T)$, it holds that
		\begin{align*}
			\| \tilde R^{\tau}_{T,v}(w)-\tilde R^{\tau'}_{T,v}(w)\| \leq C'r^{\dist(v, \Delta_{\tau, \tau'})}.
		\end{align*}
\end{lemma}
\begin{proof}
    First, let us observe that any ratio vector is contained in a compact subset $K$ of the positive quadrant in $\R^q$. 
    We see from~\eqref{eq: observation2} that for any $i\in[q]$, we have 
    \[
        1\ge \tilde R^{\tau}_{T,v;i}(w)\ge (1-(1-w))^d= w^d>0.
    \]
    
    The function $G:v\mapsto v/\|v\|_1$ is a differentiable function on the positive quadrant, thus it is $L$-Lipschitz between $(K,\|.\|_2)\to (\mathbb{R}^{q},\|.\|_\infty)$.
    Moreover, by~\eqref{eq: observation1}, for any $i\in[q]$, we have 
    \[
    G(\tilde R^{\tau}_{T,v}(w))_i=\prob_{T;w}[\phi(v)=i~|~\tau].
    \] 
    Thus, if 
    \[
        \| \tilde R^{\tau}_{T,v}(w)-\tilde R^{\tau'}_{T,v}(w)\|_2 \leq C'r^{\dist(v, \Delta_{\tau, \tau'})},
    \]
    then 
    \[
        \max_{i\in[q]} \big| \mathbb{P}_{T;w}[\Phi(v) = i|\tau] -  \mathbb{P}_{T;w}[\Phi(v) = i|\tau']\big|\le L\|\tilde R^{\tau}_{T,v}-\tilde R^{\tau'}_{T,v}\| \leq LC'r^{\dist(v, \Delta_{\tau, \tau'})},
    \]
    i.e. the condition in the claim implies SSM with decay rate $r$ and constant $C=L\cdot C'$
\end{proof}
\begin{remark}
We note that the above result crucially relies on the fact that $w>0$. 
For the case $w=0$ one would needs to make assumptions about the relation between $q$ and $d$ (e.g. $q\geq d+1$).
\end{remark}
Continuing along the same line of proof, we present the following alternative approach for WSM.
\begin{lemma}\label{lemma: ratio to WSM}
Let $d,q\in\mathbb{N}_{\geq 2}.$
If $w>0$, then the following condition implies that the Potts model at $w$ exhibits WSM with decay rate $r\in (0,1)$ on $\mathcal{T}_{d+1}$. 
There exists a constant $C'>0$ such that for any $(T,v)\in \mathcal{T}_{d+1}$, any $t \ge 1$, and any two boundary conditions $\tau, \tau': \Lambda \to [q]$, where $\Lambda =\{u\in V(T)~|~\dist(u,v)=t\}$,  it holds that
		\begin{equation*}
			\| \tilde R^{\tau}_{T,v}(w)-\tilde R^{\tau'}_{T,v}(w)\| \leq C'r^{t}.
		\end{equation*}
  \hfill\qed
\end{lemma}
% _> add into remarks file ?\khallil{Maybe a different choice than t for the distance?}
%copy pasted from MiKTeX
%file with questions and remarks

\subsection{A change of coordinates and a recursion}\label{sec: sq ratio}
To check the conditions of Lemma~\ref{lemma: ratio to SSM} and Lemma~\ref{lemma: ratio to WSM}, we will use different coordinates. 
For a tree $T=(V,E)$ with boundary condition $\tau$ and $w>0$, define the \emph{square root ratio} at a vertex $v\in V$ as
\[
    R_{T,v;i}(w)  \coloneqq \sqrt{\tilde R_{T,v;i}(w)},
\]
where $i\in[q]$.

Since the set of ratio vectors is contained in a compact subset of the open positive quadrant of $\mathbb{R}^q$, therefore, in a similar way as in Lemma~\ref{lemma: ratio to SSM} (resp. Lemma~\ref{lemma: ratio to WSM}), we obtain the following sufficient condition for SSM (resp. WSM).

\begin{lemma}\label{lemma: ratio to SSM and WSM}
Let $d,q\in \mathbb{N}_{\geq 2}$ and let $w>0$.
If there exists a constant $C>0$ and $r \in(0,1)$ such that for any rooted tree $(T,v)\in \mathcal{T}_{d+1}$, and any boundary conditions $\tau,\tau':\Lambda\to [q]$,  where $\Lambda\subseteq V(T)-\{v\}$ (resp. $\Lambda=\{u\in V(T)~|~\dist(u,v)=t\}$ for some $t\ge 1$), we have 
\begin{equation}\label{eq: sq_goal}
    \|R_{T,v}^\tau(w)-R_{T,v}^{\tau'}(w)\|^2\le Cr^{\dist(v,\Delta_{\tau,\tau'})},
\end{equation}
then the Potts model at $w$ exhibits SSM (resp. WSM) on $\mathcal{T}_{d+1}$ with decay rate $r$.
\end{lemma}

To prove our main theorems, it will be our aim to establish Equation~\ref{eq: sq_goal}. 
Throughout, we will fix positive integers $q$ and $d$.
Similarly to the ratio vectors of a tree $T$ at a vertex $v$, the square root ratio vector can also be recovered from the square root ratio vectors of the children of $v$.  By~\eqref{eq: observation3}, we can express $R_{T,v;i}(w)$ as
\begin{align}\label{eq: recursion}
    R_{T,v;i}(w) &= \prod_{j=1}^d R_{\hat{T}_{\hat{v}_j},\hat{v}_j;i}(w)= \prod_{j=1}^d F_w\left(R_{T_{v_j}, v_j ;1 }(w), \ldots, R_{T_{v_j},v_j ;q}(w)\right)_i \ ,
\end{align}
where  $F_w:\mathbb{R}^q\to \mathbb{R}^q$ is defined by
\begin{align}\label{eq: def F}
    F_w(x)_i =  \sqrt{1 + (w-1)\frac{ x_i^2}{ \sum\limits_{k \in [q]}  x_k^2}}.
\end{align}
We introduce some notation to state this more compactly. We denote
\[
S(x) = \sum_{k=1}^{q} x_k^2 \qquad \textrm{and}\qquad S_i(x) = S(x) + (w-1)x_i^2.
\]
Thus, for any $i\in[q]$, we have
\[
F_w(x)_i= \sqrt{ \frac{S_i(x)}{S(x)}}.
\]

So we see that the quantities appearing in Equation~\eqref{eq: sq_goal} are related to the square root ratios of the children of $v$ through the function $F_w$ by Equation~\eqref{eq: recursion}. 
Thus, our main goal is to establish a contraction property of $F_w(x)$. 

%It will turn out that if the colored vertices in the 1 neighborhood are the same for $\tau$ and $\tau'$, then each such a vertex helps to have a "stronger" contraction property.
%\khallil{Maybe the part above should be worded differently, such that the focus is not on the fact that tau and tau' are the same (because we can assume that anyways by induction), but just in general that we have fixed neighbours?}

Let $(T,v)\in \mathcal{T}_{d+1}$ be a tree with two boundary conditions, $\tau$ and $\tau'$, and assume that $v$ is a free vertex. 
Denote by $X_v(w)$  (resp. $X'_v(w)$) the square root ratio vector of $v$ in $T$ with the boundary condition $\tau$ (resp. $\tau'$). 
Let for $t\in [0,1]$ \[Y_v(w;t) \coloneqq tX_v(w)+(1-t)X'_v(w),\] and define the \emph{local weight} $\lambda_v(w)$ at vertex $v$ by 
\begin{equation}\label{eq:def weight}
\lambda_v(w)=\max_{i\in[q],t\in[0,1]}\dfrac{\sqrt{S(Y_v(w;t))}}{S_i\left(Y_v(w;t)  \right)}.
\end{equation}
It is not hard to see that $\lambda_v(w)^2\geq 1/q.$
The next theorem, which will be proved in Section~\ref{sec:contraction}, gives us that $F$ indeed is contracting under certain assumptions.

\begin{restatable}{theorem}{normbound}  \label{thm: norm bound}
Let $q,d\in \mathbb{N}_{\geq 2}$ and $w>0$.
Let $(T,v)\in \mathcal{T}_{d+1}$ be a rooted tree and $\Lambda\subset V(T)\setminus \{v\}$.
Consider two boundary conditions on $\Lambda$, denoted as $\tau$ and $\tau'$ such that $\dist(v,\Delta_{\tau,\tau'})\ge 2$.
Let $X_v(w)$ and $X'_v(w)$ be the square root ratio vector at $v$ for $\tau$ and $\tau'$ respectively. 
Define $\mathcal{F}$ as the set of neighbors of $v$ that are not fixed by $\tau$ and $\tau'$. 
We have the following inequality 
\[
\|X_v(w)-X'_v(w)\|^2\le \frac{1-w}{e}\sum_{k\in\mathcal{F}}\|\lambda_{v_k}(w)(X_{v_k}(w)-X'_{v_k}(w))\|^2,
\]
where $X_{v_k}(w), X'_{v_k}(w)$ represent the square root ratios of the child $v_k$ in the subtree of $T-v$ containing $v_k$ with boundary condition $\tau,\tau'$, respectively, and $\lambda_{v_k}(w)$ is the local weight at vertex ${v_k}$.
\end{restatable}

This theorem forms the spine of the inductive proof of our main results.
To be able to use it, we need a bound on the local weights.
For $a>1$ define
\begin{equation}\label{eq:def K(a)}
    K(a):=e\cdot \frac{a-1/2}{a-1}.
\end{equation}
In Section~\ref{sec: local weight} we will prove the following bound on these weights.
\begin{restatable}{proposition}{boundlambdav}\label{prop: bound lambda_v2}
Let $(T,v)\in \mathcal{T}_{d+1}$ be a rooted tree, and let $\Lambda\subset V(T)\setminus \{v\}$.
Consider two boundary conditions on $\Lambda$, denoted as $\tau$ and $\tau'$.
Let $f=|\freen|$ denote the number of free neighbors of $v$.
Assume $d\geq q+2$ and $q\geq 3$, and write $d+1=aq$ for $a>1$. 
Then for any $w\in [1-\tfrac{q}{d+1},1]$,
\begin{enumerate}
    \item[(i)] 
    \[
        \lambda_v(w)^2\le \frac{1}{q} \left(1-\frac{K(a)}{d+1}\right)^{-2}K(a)^{\tfrac{d-f}{d}}\cdot \left(e\cdot \frac{d+1-K(a)/2}{d+1-K(a)}\right)^{\tfrac{f}{d}}.
    \]
    % \[
    % \lambda_v^2 \le \frac{1}{S} \left(1-\frac{1-w}{qw^{d/q}}\right)^{-2}
    % \]
\item[(ii)] If all the neighbors of $v$ are free as well as all their respective neighbors, then
\[
 \lambda_v(w)^2 \le \frac{e}{q}\left(1+\frac{K'/2}{d+1-K'}\right)\left(1-\frac{K'}{d+1}\right)^{-2},
\]
where $K'\le \min\{q+2,13\}$ and $K'=e+O(1/d)$. 
% \item[(ii)] If all the neighbors of $v$ are free as well as all their respective neighbors, then
% \[
%  \lambda_v(w)^2 \le \frac{e}{q}\left(1+\frac{K/2}{d+1-K}\right)\left(1-\frac{\min\{9,q+2\}}{d+1}\right)^{-2},
% \]
% where $K\le \min\{q,13\}$ and $K=e+O(1/d)$. 
\end{enumerate}
\end{restatable}

 In the remainder of this section, we will use the established notation and the two results above to verify Equation~\eqref{eq: sq_goal}, thus proving our main theorems.

\subsection{Inductive proof of WSM}
In this subsection we will prove Theorem~\ref{thm:main wsm}. 
By Lemma~\ref{lemma: ratio to SSM and WSM}, it is sufficient to show that square root ratios contract with factor $r\in (0,1)$. 
Let us write $w=1-\alpha \tfrac{q}{d+1}$ with $\alpha$ to be determined later.
For notational convenience we will omit the dependence of the square root ratios on the variable $w$, i.e. we will write $X_v$ (resp. $X'_v$) instead of $X_v(w)$ (resp. $X'_v(w)$). Similarly we will write $\lambda_v$ instead of $\lambda_v(w)$.

We will prove by induction on $\dist(v,\Lambda)$ that the square root ratios contract with a factor $r\in (0,1)$. 
We claim there exists a constant $C>0$ such that
\begin{align}\label{eq: wsm goal}
\|X_v-X'_v\|^2\le C(\tfrac{d}{d+1})^{\dist(v,\Lambda)},
\end{align}
thus implying that the decay rate is $\tfrac{d}{d+1}.$
Since there are finitely many boundary conditions at a given distance, we may assume that $\dist(v,\Lambda)>3$ by choosing $C$ sufficiently large. 
%In particular, this ensures that the neighbors of $v$ are all free and that their neighbors are free as well. 
Additionally, this deals with the base case of the statement in \eqref{eq: wsm goal}.

Now, let us assume the induction hypothesis, which implies that for each neighbor $v_k$ of $v$,
\begin{align}
 \|X_{v_k}-X'_{v_k}\|^2\le C(\tfrac{d}{d+1})^{\dist(v_k,\Lambda)}.
\end{align}
Note that $\dist({v_k},\Lambda) = \dist(v, \Lambda)-1$. 
Denote the number of neighbors of $v$ by $d'$.
Combining the induction hypothesis with Theorem~\ref{thm: norm bound} gives us 
\begin{align*}
\|X_v-X'_v\|^2&\le \frac{(1-w)}{e}\sum_{k=1}^{d'} \lambda_{v_k}^2 Cr^{\dist(v,\Lambda)-1} \leq \frac{\alpha q}{e}\frac{d}{d+1}\max_{k=1,\ldots,d'}\lambda_{v_k}^2 C\left(\tfrac{d}{d+1}\right)^{\dist(v,\Lambda)-1}.
\end{align*}

To complete the induction, it suffices to show that 
\begin{align}\label{eq:required bound WSM}
    \frac{\alpha q}{e}\lambda_{v_k}^2\leq 1
\end{align}
for all $k=1,\ldots, d'$.

By assumption, the neighbors of $v_k$ and their respective neighbors are all free.
So by applying Proposition~\ref{prop: bound lambda_v2}, we conclude that there exists a constant $K'=e+O(1/d)\le \min\{13,q+2\}$ such that for $k=1,\ldots,d',$
\[
    \frac{\alpha q}{e}\lambda_{v_k}^2\le \alpha\left(1+\frac{K'/2}{d+1-K'}\right)\left(1-\frac{K'}{d+1}\right)^{-2}\le 1,
\]
provided that $\alpha^{-1}=\left(1+\frac{K'/2}{d+1-K'}\right)\left(1-\frac{K'}{d+1}\right)^{-2}$. 
Note that this choice of $\alpha$ satisfies $\alpha=(1-K_q''/d)$ for some constant $K_q''=5e/2+O(1/d)\le 40$.
% \feri{rewrite from here}

% Recall from~\eqref{eq: def B} the definition of $B(0)$. 
% For shorthand notation, let $B=B(0)$ and $\tilde B=(1-(1-w)B)^{-\frac{1}{qB}}$. 
% Observe that $1/q\le B\le 1$, thus \[(1-(1-w)/q)^{-1}\le \hat B \le w^{-1/q}.\]

% Using Corollary~\ref{cor: bounds on B(d)}, Lemma~\ref{lemma: bound on S_n}, and Lemma ~\ref{lem:bound lambda_v2} gives us the upper bound
% \begin{align*}
%     \lambda_v^2  &\leq \frac{1}{S} \left( 1 - \frac{1-w}{S} \right)^{-2}\\
%     &\leq \frac{e}{q} \frac{d+1 -\frac{B}{2q}}{d+1-B/q}\left( 1- \frac{e(1-w)}{q} \frac{d+1 - \frac{B}{2q}}{d+1-B/q}  \right)^{-2} \\
%     &\leq \frac{e}{q} \frac{d+1 -K/2}{d+1-K}\left( 1- \frac{e }{d+1} \frac{d+1 - K/2}{d+1-K}  \right)^{-2}, 
% \end{align*}

% with $K=1 + O(1/d)$ and using the fact that the penultimate expression is increasing in $B$.

% By using the upper bound for $\lambda_v^2$ and choosing $\alpha$ such that 
% \[
% \alpha \leq \left(  \frac{d+1 -K/2}{d+1-K}\left( 1- \frac{e }{d+1} \frac{d+1 - K/2}{d+1-K}  \right)^{-2} \right) ^{-1},
% \]
% the LHS of ~\eqref{eq:required bound WSM} reduces to $d/(d+1)$. Thus, for $\delta \leq 1/(d+1)$, ~\eqref{eq:required bound WSM} holds, and the system exhibits weak spatial mixing.
\begin{remark}
The contraction rate of $\frac{d}{d+1}$ has been chosen to simplify the calculations above. 
We note that choosing a rate closer to $1$ only has a small effect if $d$ is large.
It would allow us to replace the constant $K$ by $K-1$, however when  $d$ is small, we expect to have better bounds on the quantities $\lambda_{v_k}$ cf.~Section~\ref{sec:conclusion}.
\end{remark} 
\subsection{Inductive proof of SSM}
In this subsection we will prove Theorem~\ref{thm:main ssm}. 
Similarly to WSM, by Lemma~\ref{lemma: ratio to SSM and WSM}, it is sufficient to show that square root ratios contract with factor $r=\tfrac{d}{d+1}$.
We again let $w=1-\alpha\frac{q}{d+1}$, with $\alpha$ to be determined later.
As above for notational convenience we will omit the dependence of the square root ratios on the variable $w$, i.e. we will write $X_v$ (resp. $X'_v$) instead of $X_v(w)$ (resp. $X'_v(w)$). Similarly we will write $\lambda_v$ instead of $\lambda_v(w)$.

We will prove by induction on $\dist(\Delta_{\tau,\tau'},v)$ that the square-root ratios contract.
More precisely we claim there exists a constant $C>0$ such that 
\begin{align}\label{eq: ssm goal}
\|X_v-X'_v\|^2\le |\mathcal{F}|C(\tfrac{d}{d+1})^{\dist(v,\Delta_{\tau,\tau'})},
\end{align}
where $\mathcal{F}$ denotes the set of free neighbors of $v$.
Since there are only finitely many configurations that could occur in a fixed depth neighborhood, we may assume that $\dist(v,\Delta_{\tau,\tau'})>2$ by choosing $C$ sufficiently large. In other words, by choosing $C$ large enough, we establish the base cases $\dist(v,\Delta_{\tau,\tau'})\in\{0,1,2\}$.

Let us combine Theorem~\ref{thm: norm bound}  with the induction hypothesis to bound $\|X_v-X'_v\|^2$.
We get
\begin{align*}\label{eq: strong ssm}
\|X_v-X'_v\|^2&\le \frac{(1-w)}{e}\sum_{k\in\freen}f_k\lambda_{v_k}^2 C(\tfrac{d}{d+1})^{\dist(v,\Delta_{\tau,\tau'})-1},
\end{align*}
where $f_k$ denotes the number of free neighbors of vertex $v_k$ in the tree $T-v$.
It thus suffices to show that
\[
\frac{(1-w)}{e}\max_{k\in \mathcal{F}}f_k\lambda_{v_k}^2\leq \tfrac{d}{d+1}.
\]
% We do this separately for the general case $d+1\geq q$ and the special case $d+1\geq \tfrac{e-1/2}{e-1}q$.
% In the general case, applying Lemma~\ref{lem:bound lambda_v2} with $f=0$ we obtain, writing $d+1=aq$ and recalling that $K(a)=e\tfrac{d+1-q/2}{d+1-q}$, we have
% \begin{align*}
% f\frac{(1-w)}{e}\lambda_v^2\leq& \frac{\alpha d}{d+1}\left(1-\frac{K(a)}{(d+1)}\right)^{-2}\frac{d+1-q/2}{d+1-q},
% \end{align*}
% which is bounded by $1-1/(d+1)$ provided we choose $1/\alpha=\left(1-\frac{K(a)}{(d+1)}\right)^{-2}\frac{d+1-q/2}{d+1-q}$.
% Noting that for fixed $q$ we have $\alpha=1-K'/d$ for some constant $K'$ this finishes the proof of the general case.

Writing $d+1=aq$, we have by Proposition~\ref{prop: bound lambda_v2}, for any $k\in \mathcal{F}$,
\begin{align}
 f_k\frac{(1-w)}{e}\lambda_{v_k}^2 &\le\frac{\alpha\cdot f_k}{e(d+1)}\left(1-\frac{K(a)}{(d+1)}\right)^{-2}(K(a))^{(d-f_k)/d}\left(e\cdot \frac{d+1-K(a)/2}{d+1-K(a)}\right)^{f_k/d}\nonumber
 \\
 &=\frac{\alpha d}{d+1}\frac{K(a)}{e}\cdot\left(1-\frac{K(a)}{(d+1)}\right)^{-2} \frac{f_k}{d}\left(e\cdot \frac{(d+1)/K(a)-1/2}{d+1-K(a)}\right)^{f_k/d}.  \label{eq:function in f}
% &\le \left(1-\frac{8e}{3(d+1)}\right)^{-2}\frac{1-w}{e}\frac{1}{q}w^{-(d-f)/q}\hat{B}^{f}f C(1-\delta)^{\dist(\Delta_{\tau,\tau'})-1},
\end{align}

%I think there are some mistakes, should double check
We claim that the right hand side of~\eqref{eq:function in f} is maximal when $f_k/d=1$. 
Indeed, the function $x\mapsto Axb^x$ has derivative $Ab^x(x\log b+1)$ and thus its critical point is at $x=-1/\log(b)$, which is larger than $1$ provided $b\geq 1/e$.
%\khallil{Actually also provided $b <1$, but if $b<1$ then the function is clearly increasing on $[0,1]$. Mention this or no need?}
Since $a\geq \frac{e-1/2}{e-1}$, we have $K(a)\leq e^2$, and therefore 
\[
(d+1)/K(a)-1/2\geq (d+1-K(a))/e^2.
\]
Indeed, the claim is equivalent to $(d+1)(1/K(a)-1/e^2)\geq 1/2 -K(a)/e^2$, which is automatically true if $K(a)\geq e^2/2$, while if $K(a)\leq e^2/2$ this also follows, since $d+1\geq e^2/2$.
Thus, we conclude that $\left(e\cdot \frac{(d+1)/K(a)-1/2}{d+1-K(a)}\right)\geq 1/e$ and therefore~\eqref{eq:function in f} is indeed maximized at $f_k=d$.

Now, we can plug in $f_k=d$  (and the upper bound of $e^2$ on $K(a)$) into~\eqref{eq:function in f} to conclude that for any $ k\in\mathcal{F}$,
\[
f_k\frac{(1-w)}{e}\lambda_{v_k}^2\leq \alpha \left(1-\frac{e^2}{d+1}\right)^{-2} \cdot \left(\frac{d+1-e^2/2}{d+1-e^2}\right)\frac{d}{d+1}=\frac{d}{d+1},
\]
provided we choose $\alpha^{-1}=\left(1-\frac{e^2}{d+1}\right)^{-2} \cdot \left(\frac{d+1-e^2/2}{d+1-e^2}\right)$.
Noting that this choice of $\alpha$ satisfies $\alpha=(1-K/d)$ for some constant $K$ independent of $q$, this finishes the proof.

\begin{remark}\label{rem:SSM 2}
In case $a<\tfrac{e-1/2}{e-1}$ we can still carry through the induction, but we must then choose $\alpha$ such that~\eqref{eq:function in f} is at most $d/(d+1)$. It is unclear whether the $K$ in the resulting $\alpha=1-K/d$ can be chosen to be independent of $q$.
\end{remark}

\section{A bound on the strength of the contraction}\label{sec:contraction}
It is our goal to prove Theorem~\ref{thm: norm bound} in this section. 
For this, we recall some setup and notation that will be kept fixed and used throughout this section. To this we fix $q,d\in \mathbb{N}_{\geq 2}$ and $w>0$.
Let $(T,v)\in \mathcal{T}_{d+1}$ be a rooted tree and let $\Lambda\subset V(T)\setminus \{v\}$. 
Let $\tau, \tau'$ be two boundary conditions on $\Lambda$, such that $\dist(v,\Delta_{\tau,\tau'})\ge 2$.
For notational convenience, in what follows we will omit the dependence on $w$ for the square root ratios and other functions derived from these square root ratios. 
We will also write $F$ shorthand for $F_w$. 

Denote the square root ratio of $v$ in $T$ for $\tau$ (resp. $\tau'$) by $X_v$ (resp. $X'_v$), and let 
\[
Y_v(t)=tX_v+(1-t)X'_v.
\] 
Denote by $v_1,\ldots,v_{d'}$ the neighbors of $v$ (and note that $d'\leq d$).
%Let $\freen$ be the set of free neighbors of $v$ in $T$. 
For each $j=1,\ldots,d'$, let $X_{v_j}$ (resp. $X'_{v_j}$) be the square root ratios of $v_j$ in the component of $T-v$ containing $v_j$ for the boundary condition $\tau$ restricted to this component (resp. for $\tau'$). 
Define
\begin{align*}
    Y_{v_j}(t)=tX_{v_j} + (1-t)X'_{v_j}.% \qquad  \lambda_{v_j}^2=\max_{i\in [q],t\in[0,1]}\frac{S(Z_j(w;t))}{(S_i(Z_j(w;t))^2}.
\end{align*}

%This section is devoted to prove the following theorem under the established notations.

% \begin{theorem}
%     We have
%     \[
%     \|X-Y\|^2\le \frac{1-w}{e}\sum_{k\in\freen}\|\lambda_{v_k}(X_k-Y_k)\|^2.
%     \]
% \end{theorem}

By \eqref{eq: recursion} and \eqref{eq: def F} we have the relationships 
\begin{align*}
    X_v=  \prod_{j=1}^{d'} F( X_{v_j}), \quad X'_v =  \prod_{j=1}^{d'} F( X'_{v_j}),
\end{align*}
where $\prod$ denotes the coordinate-wise product of vectors.
% therefore we will aim to show that $F$ is a contraction. 
% This so called contractive property is satisfied when $\| D\hat{F} \|_2 <1$ (double check?), since
% \begin{align*}
    % X = \frac{1}{d} \sum_{j=1}^d \hat{F}( X_j), \quad Y = \frac{1}{d} \sum_{j=1}^d \hat{F}(Y_j).
    % \end{align*}
Suppose that $v_k$ is fixed for some $k$. 
Since $v_k$ has distance one
%\feri{$< 2$, right?} 
from $v$, by assumption $\tau(v_k) = \tau'(v_k)$. 
Therefore, it also follows that $X_{v_k}= X'_{v_k}$, and in particular $F(X_{v_k})=F(X'_{v_k})$.
Moreover, $0 \leq F(X'_{v_k})_i \leq 1$ for all $i$, by the definition of $F$.
Therefore, for each $i=1,\ldots, q,$
\begin{align*}
    |X_{v,i} - X'_{v,i} | &= \Big| \prod_{j=1}^d F( X_{v_j})_i -  \prod_{j=1}^d F(X'_{v_j})_i \Big|\\
  %  &= \Bigg\|\Bigg( F(X_k)  \prod_{\substack{j=1\\ j \neq k}}^d F( X_j) -  F(Y_k) \prod_{\substack{j=1\\ j \neq k}}^d F(Y_j) \Bigg) \Bigg\|\\
    &= \Bigg|F(X_{v_k})_i \Bigg( \prod_{\substack{j=1\\ j \neq k}}^d F( X_{v_j})_i -  \prod_{\substack{j=1\\ j \neq k}}^d F(X'_{v_j})_i \Bigg) \Bigg|\\
    &\leq  \Big| \prod_{\substack{j=1\\ j \neq k}}^d F( X_{v_j})_i -  \prod_{\substack{j=1\\ j \neq k}}^d F(X'_{v_j})_i \Big|.
\end{align*}
%We see that fixed neighbours therefore only give an improvement, and therefore it is justified to assume from now on that $v$ has no fixed neighbours. 
For later purposes, let us denote by $\freen=\{j\in [d]~|~\textrm{$v_j$ is not fixed}\}$. 
By the reasoning above, we therefore have for each $i=1,\ldots,q$,
\begin{equation}\label{eq:reduce to free}
    |X_{v,i}-X'_{v,i}|\le \Big|\prod_{j\in \freen}F(X_{v_j})_i-\prod_{j\in \freen}F(X'_{v_j})_i\Big|.
\end{equation}

\begin{lemma} \label{lemma: change to integral}
    Using the notation given above, we have
\begin{equation}\label{eq:int representation}
    \| X_v - X'_v \|^2 \leq \int_0^1  \Bigg\|\frac{\partial}{\partial t}  \Big(\prod_{j\in\freen} F\left( Y_{v_j}(t) \right)\Big) \Bigg\|^2 \, dt.
\end{equation}
    %where  $Z_j(t) = tX_j + (1-t)Y_j$ for each $j\in \freen$.
\end{lemma}
\begin{proof}
    %Let $X_j = \phi^{-1}(U_j)$ and $Y_j = \phi^{-1}(V_j)$.
By~\eqref{eq:reduce to free}, the $i$-th coordinate of $X_v-X'_v$ satisfies
    \begin{align*}
        (X_{v,i} - X'_{v,i})^2  &\leq\Big(\prod_{j\in \freen} F(X_{v_j})_i -  \prod_{j\in\freen} F(X'_{v_j} )_i\Big)^2 \\
          &=\left(\int_0^1 \frac{\partial}{\partial t}  \Big(\prod_{j\in\freen} F\left(  tX_{v_j} + (1-t)X'_{v_j} \right)_i\Big)   \, dt\right)^2\\
        &\leq\int_0^1 \Big(\frac{\partial}{\partial t}  \Big(\prod_{j\in\freen} F\left( Y_{v_j}(t) \right)_i\Big)\Big)^2   \, dt,
    \end{align*}
    by Jensen's inequality.
    %where  $Z_j(w;t) =tX_j + (1-t)Y_j $.
 This implies~\eqref{eq:int representation}.
    %\begin{align*}
     %   \| X - Y \| &\leq \int_0^1  \left\|\frac{\partial}{\partial t}  \left(\prod_{j\in\freen} F\left( Z_j(w;t) \right)\right) \right\| \, dt.
    %\end{align*}
\end{proof}

In order to bound the norm appearing in the integral for each $t\in [0,1]$, we will first give a factorization of the Jacobian of $F$, $DF$.
\begin{lemma}\label{lemma: DF factorization}
The Jacobian of $F$ satisfies
    \[
    DF(x) = \diag\left(  \left( \frac{w-1}{F(x) } \right)  \circ \left( \frac{x}{S(x)} \right)\right)   \left[ I  -  \left( \frac{x}{\sqrt{S(x)}} \right) \left( \frac{x}{\sqrt{S(x)}} \right)^T \right].
    \]  
    Here, for a vector $y$, we use $\diag(y)$ to represent the $ q \times q $ diagonal matrix with the entries of $y $ along its diagonal. 
    The symbol $\circ$ denotes the Hadamard  product of vectors. 
    Additionally, $ 1/F(x)$ denotes the vector with $i-$th coordinate $1/F(x)_i$. 
\end{lemma}

\begin{proof}
    First,  we will compute the partial derivatives of $F$. 
    We find for $i \neq j$ that
    \begin{align*}
        \pdv{F(x)_i }{x_j} =  \dfrac{1}{2 \sqrt{ \dfrac{S_i(x)}{S(x)} }} \cdot  (w-1)\left( 0- \frac{2 x_i^2 x_j  }{S(x)^2}  \right) = \frac{1}{F(x)_i} \cdot (w-1) \frac{-x_i^2 x_j}{S(x)^2},
    \end{align*}
    and
    \begin{align*}
        \pdv{F(x)_i }{x_i} =\dfrac{1}{2 \sqrt{ \dfrac{S_i(x)}{S(x)} }} \cdot  (w-1)\left(\frac{ 2x_i S(x)- 2 x_i^3  }{S(x)^2}  \right)  = \dfrac{1}{F(x)_i} \cdot  (w-1)\left(\frac{ x_i S(x) -  x_i^3  }{S(x)^2}  \right) .
    \end{align*}
    We obtain:
    \begin{align*}
        DF(x) & =  (w-1) \diag\left( \frac{1}{F(x)} \right) \left[ \diag \left( \frac{x}{S(x)}  \right)   - \left(\frac{(x\circ x) x^T}{S(x)^2} \right) \right]\\
        &=(w-1) \diag\left( \frac{1}{F(x)} \right) \left[ \diag \left( \frac{x}{S(x)}  \right)   -  \left( \frac{x\circ x}{S(x)} \right) \left( \frac{x}{S(x)} \right)^T \right]\\
        &= (w-1) \diag\left( \frac{1}{F(x)} \right) \diag \left( \frac{x}{S(x)}  \right)  \left[ I  -  x  \left( \frac{x}{S(x)} \right)^T \right]\\
        & =\diag\left( \left( \frac{w-1}{F(x) }  \right) \circ \left( \frac{x}{S(x)} \right)\right)   \left[ I  -  \left( \frac{x}{\sqrt{S(x)}} \right) \left( \frac{x}{\sqrt{S(x)}} \right)^T \right].
    \end{align*}
    
\end{proof}
	
	From~\cite{chen2023strong} we have the following lemma, that we will use to bound the norm appearing in the integral of Lemma~\ref{lemma: change to integral}.
	\begin{lemma}
		For any positive integers $d,q \geq 1$, diagonal matrices $D_1, \ldots, D_d \in \R^{q\times q}$, and vectors $\bm{x}_1, \ldots, \bm{x}_d \in \R^q$, we have 
		\begin{align*}
			\Big\| \sum_{j=1}^d D_i \bm{x}_j \Big\|_2^2    \leq \max_{i \in [q] }\sum_{j=1}^d D_j(i,i)^2\cdot  \sum_{j=1}^d \|\bm{x_}j \|_2^2.
		\end{align*}
		\label{B1Chen}
	\end{lemma}
	
	\begin{lemma}\label{lemma: bound on the integral}
		For any $t\in[0,1]$ we have
		\begin{align}
			&\left\|\frac{\partial}{\partial t} \prod_{j\in\freen} F\left( Y_{v_j}(t) \right) \right\|^2 \le \nonumber
   \\
   \quad & \max_{ i \in [q]}\left( \sum_{k\in\freen} 
			\left[\prod_{j \in\freen} F \left( Y_{v_j}(t) \right)_i^2 \right]  \frac{(1-w)^2 Y_{v_k,i}^2}{S(Y_{v_k}) }   \right) \cdot  \sum_{k\in\freen}\norm{  \lambda_{v_k}  \left(X_{v_k}-X'_{v_k} \right)  }^2. \label{eq: bound on the integral}
		\end{align}
	\end{lemma}
	\begin{proof}
		
		Notice that the product in the LHS of~\eqref{eq: bound on the integral} is component-wise. 
		Using the product rule for the $i$-th coordinate, we get that
		\begin{align*}
			\frac{\partial}{\partial t}   \prod_{j\in\freen} F\left( Y_{v_j}(t) \right)_i & = \sum_{k\in\freen}  \pdv{F \left( Y_{v_k}(t) \right)_i}{t} \prod_{\substack{j\in\freen\\j\neq k}} F \left( Y_{v_j}(t) \right)_i.
		\end{align*}
		Now we focus on $\pdv{}{t} F(Y_{v_k}(t))$. Using the chain rule, we obtain
		\begin{align*}
			\pdv{F \left( Y_{v_k}(t) \right)}{t} & = DF(Y_{v_k}(t)) \cdot \left( X_{v_k} - X'_{v_k} \right).\\
		\end{align*}
		This implies that the LHS of~\eqref{eq: bound on the integral} becomes
		\begin{align*}
			&\left\|\sum_{k\in\freen} \bigg( DF(Y_{v_k}(t)) \cdot \left( X_{v_k} - X'_{v_k}  \right) \bigg)\circ \prod_{\substack{j\in\freen\\j \neq k}} F \left( Y_{v_j}(t) \right)  \right\|^2  \\ =&    \left\|\sum_{k\in\freen}  \diag\left( \prod_{\substack{j\in\freen\\j \neq k}} F \left( Y_{v_j}(t) \right) \right) \cdot DF(Y_{v_k}(t)) \cdot \left( X_{v_k} - X'_{v_k}  \right)    \right\|^2, 
		\end{align*}

  In what follows we will omit the variable $t$ from the notation and just write $Y_{v_k}$ instead of $Y_{v_k}(t)$.
		From Lemma~\ref{lemma: DF factorization} we know that   $DF(Y_{v_k}(t)) \cdot \left( X_{v_k} - X'_{v_k}  \right)$ is equal to
		\begin{align*}
\diag\left( \left( \frac{w-1}{F(Y_{v_k}) }  \right) \circ  \left( \frac{Y_{{v_k}}}{S(Y_{v_k})}\right) \right)  \left[ I  -  \left( \frac{Y_{{v_k}}}{\sqrt{S(Y_{v_k})}} \right) \left( \frac{Y_{{v_k}}}{\sqrt{S(Y_{v_k})}} \right)^T \right]\cdot \left( X_{v_k} - X'_{v_k}  \right).
		\end{align*}
		
		Let us denote by $\pi(Y_{v_j})_i = Y_{{v_j},i} \big/\sqrt{S(Y_{v_j})}$, and $\alpha_{v_j} = \left[ I - \pi(Y_{v_j}) \pi(Y_{v_j})^T\right] \cdot( X_{v_j} - X'_{v_j}) $. 
		Using these notations, we obtain that the  LHS of~\eqref{eq: bound on the integral}  is equal to
		
		% \begin{align*}
		% 	&\left\|\sum_{k\in\freen}  \diag\left( \prod_{j \neq k} F \left( Z_j(t) \right)_i \right)  \diag\left( \frac{w-1}{F(Z_k)_i }  \cdot \frac{Z_{k,i}}{S(Z_k)}\right)   \alpha_k   \right\|^2 \\
		% 	=&\left\|\sum_{k\in\freen}  \diag\left( \prod_{j\in\freen} F \left( Z_j(t) \right)_i \right)  \diag\left( \frac{w-1}{F(Z_k)_i^2 }  \cdot \frac{Z_{k,i}}{S(Z_k)}\right)   \alpha_k   \right\|^2 \\
		% 	=&\left\|\sum_{k\in\freen}^d  \diag\left( \prod_{j \in\freen}^d F \left( Z_j(t) \right)_i \right)  \diag\left( \frac{(w-1) Z_{k,i}}{ \sqrt{S(Z_k)} }  \cdot \frac{1}{ \sqrt{S(Z_k)} F(Z_k)_i^2 }\right)   \alpha_k   \right\|^2 \\
		% 	=&\left\|\sum_{k\in\freen}^d  \diag\left( 
		% 		\left[\prod_{j \in\freen}^d F \left( Z_j(t) \right)_i \right]  \frac{(w-1) Z_{k,i}}{ \sqrt{S(Z_k)} } \right) \diag\left( \frac{\sqrt{S(Z_k)}}{S_i(Z_k) }\right)   \alpha_k   \right\|^2 .
		% \end{align*}

  \begin{align*}
&\left\|\sum_{k\in\freen} \diag\left( \prod_{\substack{j\in\freen\\j \neq k}} F \left( Y_{v_j} \right) \right) \diag\left( \left(\frac{w-1}{F(Y_{v_k})} \right) \circ \left( \frac{Y_{v_k}}{S(Y_{v_k})}\right) \right) \alpha_{v_k} \right\|^2 \\
=&\left\|\sum_{k\in\freen} \diag\left( \prod_{j\in\freen} F \left( Y_{v_j} \right) \right) \diag\left( \left( \frac{w-1}{F(Y_{v_k})^2}  \right) \circ \left( \frac{Y_{v_k}}{S(Y_{v_k})}\right)  \right)\alpha_{v_k} \right\|^2 \\
=&\left\|\sum_{k\in\freen} \diag\left( \prod_{j \in\freen} F \left( Y_{v_j} \right) \right) \diag\left( \left( \frac{(w-1) Y_{v_k}}{\sqrt{S(Y_{v_k})}}  \right) \circ \left( \frac{1}{\sqrt{S(Y_{v_k})} F(Y_{v_k})^2}\right) \right) \alpha_{v_k} \right\|^2 \\
=&\left\|\sum_{k\in\freen} \diag\left( \left[\prod_{j \in\freen} F \left( Y_{v_j} \right) \right] \circ \left( \frac{(w-1) Y_{v_k}}{\sqrt{S(Y_{v_k})}} \right) \right) \diag\left( \frac{\sqrt{S(Y_{v_k})}}{S_i(Y_{v_k})}\right) \alpha_{v_k} \right\|^2 ,
\end{align*}

where $\diag\left( \sqrt{S(Y_{v_k})}\big/ S_i(Y_{v_k})\right)$ denotes the $q \times q$ diagonal matrix with $\sqrt{S(Y_{v_k})}/S_i(Y_{v_k})$ on the $i-$th diagonal entry.
The last step of the computation follows from the identity $F(Y_{v_k})_i^2 = S_i(Y_{v_k}) \big/ S(Y_{v_k}).$
Applying Lemma \ref{B1Chen}, we can bound this by

  \begin{align}
%&\left\|\sum_{k\in\freen} \diag\left( 
%\left[\prod_{j \in\freen} F \left( Y_{v_j} \right)_i \right]  \frac{(w-1) Y_{v_k,i}}{ \sqrt{S(Y_{v_k})} } \right) \diag\left( \frac{\sqrt{S(Y_{v_k})}}{S_i(Y_{v_k}) }\right) \alpha_{v_k} \right\|^2 \\
%\stackrel{\text{Lemma \ref{B1Chen}}}{\leq}
& \max_{ i \in [q]}\left\{ \sum_{k\in\freen} \left( 
\left[\prod_{j \in\freen} F \left( Y_{v_j} \right)_i \right]  \frac{(w-1) Y_{v_k,i}}{ \sqrt{S(Y_{v_k})} } \right)^2 \right\} \sum_{k\in\freen} \left\| \diag\left( \frac{\sqrt{S(Y_{v_k})}}{S_i(Y_{v_k}) }\right) \alpha_{v_k} \right\|^2 \nonumber \\
\leq &\max_{ i \in [q]}\left( \sum_{k\in\freen}
\left[\prod_{j \in\freen} F \left( Y_{v_j} \right)_i^2 \right]  \frac{(1-w)^2 Y_{v_k,i}^2}{S(Y_{v_k}) } \right) \sum_{k\in\freen} \norm{  \lambda_{v_k}  \alpha_{v_k}  }^2 .\label{eq: alpha vector bound}
\end{align}
Recalling the definition of $\lambda_{v_k}$ from ~\eqref{eq:def weight}, we have that
  \[
		\max_{i\in [q]}\frac{\sqrt{S(Y_{v_k}(t))}}{S_i(Y_{v_k}(t))}\leq \lambda_{v_k}.
		\]

Since $\pi(Y_{v_j})$ has norm $1$, it follows that $I - \pi(Y_{v_j}) \pi(Y_{v_j})^T$ has operator norm at most $1$. 
We can therefore further upper bound the expression in~\eqref{eq: alpha vector bound} by
\begin{equation*}
\max_{i \in [q]}\left\{ \sum_{k\in\freen} \left( 
\left[\prod_{j \in\freen} F \left( Y_{v_j}(t) \right)_i^2 \right] \frac{(1-w)^2 Y_{v_k,i}^2}{S(Y_{v_k}) } \right) \right\} \cdot \sum_{k\in\freen} \norm{ \lambda_{v_k} \left(X_{v_k}-X'_{v_k} \right) }^2,
\label{eq: RHS bound on integral}
\end{equation*}
finishing the proof.

	\end{proof}

 We next bound the maximum, after extracting one $(1-w)$ factor, appearing in Equation~\ref{eq: bound on the integral}.

	\begin{lemma}\label{lem: beta bound}
		Using the previous notation, for any $t\in[0,1]$ we have 
		\begin{equation}
			\left[ \prod_{j \in\freen} F \left( Y_{v_j}(t) \right)_i^2 \right] \cdot   \sum_{k\in\freen}  \frac{(1-w)Y_{v_k,i}(t)^2}{S(Y_{v_k}(t))}  \leq e^{-1}.
			\label{eq: beta bound}
		\end{equation}
		\end{lemma}
	
	\begin{proof}
		Set $ \beta_{v_j} \coloneqq (1-w) \frac{Y_{v_j,i}(t)^2}{S(Y_{v_j}(t))}$. 
		Note that $F(Y_{v_j}(t))_i^2 = 1-\beta_{v_j}$.
		The LHS of ~\eqref{eq: beta bound} becomes
		\[
		\prod_{j\in\freen}(1-\beta_{v_j}) \sum_{k\in\freen} \beta_{v_k}.
		\]
		
		Note that $\beta_{v_j} >0$ and $1-\beta_{v_j}\ge 0$ for every $j\in\freen$.
  % because $Y_{v_j}$ is a convex combination of square root ratios, and hence non-negative.
  	We upper bound the expression with AM-GM to obtain
		\begin{align*}
			\prod_{j\in\freen} (1- \beta_{v_j})  \cdot \sum_{k\in\freen} \beta_{v_k} \le \left(\frac{|\freen|}{|\freen|+1}\right)^{|\freen|+1}\le \left(\frac{d}{d+1}\right)^{d+1}\le e^{-1}.
		\end{align*}
		
	\end{proof}
	
By combining the previous two lemmas, we obtain
\[
    \|X_v-X'_v\|^2\le\frac{1-w}{e}\sum_{k\in\freen}\|\lambda_{v_k}(X_{v_k}-X'_{v_k})\|^2,
\]
and thereby proving Theorem~\ref{thm: norm bound}.

\section{Bounds on the local weights } \label{sec: local weight}
To proceed, we have to establish a bound on the local weights $\lambda_v^2$. 
Since the definition of $\lambda_v$ involves $S=S(Y_v(w;t))$, which in turn depends on the marginal distribution of the root vertex, we first spend some time deriving suitable bounds for these.

\subsection{Bounds on the marginal of the root vertex}
The following lemma is very similar to \cite{chen2023strong}[Lemma~B.2] and will be used to prove our bounds below.
% \feri{This one doesn't follow right? or I'm just overlooking something? Ahh... oK it seems that the proof is the same, they just overlooked the condition...}
\begin{lemma}\label{lemma: Bernoulli opt}
    Assume that for any $i=1,\dots,q$, we have $0\le x_i\le b\le 1$ and $\sum_{i=1}^q x_i=1$. 
    Then, for any $\alpha\in [0,1]$, we have
    \begin{equation}\label{eq:product bound}
        \prod_{i=1}^q(1-\alpha x_i) \ge (1-\alpha b)^{1/b}.
    \end{equation}
\end{lemma}

\begin{proof}
Assume that $\{x_i\}_{i=1}^q$ is a minimizer of the left hand side. By symmetry we may assume that $0\le x_1\le \dots\le x_q\le b$.

First consider the function
\[
    g(t)=(1-\alpha (x_i-t))(1-\alpha (x_j+t)),
\]
where $x_i\le x_j$.
Then $g(t)$ is a quadratic function in $t$ with a negative leading coefficient and a negative linear term, and therefore $g(t)$ is monotonically decreasing for $t\ge 0$. 
Thus, if $0\neq x_i\le x_j\neq b$, then there exists $\{x_i'\}_{i=1}^q$ that is better, contradicting the choice of the $x_i$'s.
This means that the minimum of the LHS of ~\eqref{eq:product bound} is achieved when 
\[
0=x_1=\dots=x_\ell<x_{\ell+1}\le x_{\ell+2}=\dots=x_q=b,
\]
where $\ell= q- \lfloor 1/b\rfloor$ and $x_{\ell+1}=1-(q-\ell-1)b$.

By the Bernoulli-inequality we know that
\[
    (1-\alpha x_{\ell+1})\ge (1-\alpha b)^{x_{\ell+1}/b}=(1-\alpha b)^{1/b-(q-\ell-1)},
\]
and thus
\[
    \prod_{i=0}^q (1-\alpha x_i)\ge (1-\alpha b)^{1/b},
\]
proving the lemma.
\end{proof}
Define for $\ell=0,\ldots,d$ and $w\in[0,1]$
\begin{align}
M(w):&=\left(1-\frac{1-w}{1+(q-1)w^{d/(q-1)}}\right)^{{1+(q-1)w^{d/(q-1)}}}, \quad \text{and} \nonumber
    \\
    B(\ell):&=\frac{1}{1+(q-1)w^{\ell/(q-1)}M(w)^{(d-\ell)/(q-1)}}.\label{eq: def B}
\end{align}
Note that $M(w)>w$ by Bernoulli's inequality, thus $B(0)< B(1)< \dots < B(d)< 1$.

The next result provides a bound for the marginal probabilities of the root vertex. Item (i) is standard but we include it (with a proof) for completeness. Item (ii) follows along the same lines and is likely known too, but we don't know of an explicit reference for it.
\begin{lemma}\label{lem:prob basic}
    Let $(T,\Lambda,\tau)$ be a partially $q$-colored tree and let $v$ be a free vertex of $T$ of degree $d$. 
    Then
    \begin{enumerate}[label = (\roman*) ]
      \item (One step bound) for any $w\in [0,1]$ and any color $i$, it holds that
      	\begin{align*}
    			\prob_{T;w}[\Phi(v)=i\mid \tau]\leq B(d)=\frac{1}{1+(q-1)w^{d/(q-1)}}.
    		\end{align*}
      
      \item (Two step bound) Additionally, assume that exactly $\ell$ number of neighbors of $v$ are precolored and all the free neighbors of $v$ have degree at most $d+1$. 
      Then for any $w\in [0,1)$ and any color $j$, 
        \[
        \prob_{T;w}[\Phi(v)=i\mid \tau]\le B(\ell).
        \]  
\end{enumerate}
\end{lemma}
\begin{proof}[Proof of (i).]\let\qed\relax
In what follows, we will write $\prob_{T}$ instead of $\prob_{T;w}$.		
  We may assume that $i=1$.
		%We can expand $\prob_{T}[\Phi(v)=1\mid \tau ]$ over the colorings $\kappa$ of the neighbors of $v$.
		%This yields
	%	\begin{align*}
	%		\prob_T[\Phi(v)=1\mid]=\sum_{\kappa:N(v)\to [q]}\prob_T[v=1 \mid \kappa]\prob_T[\kappa\mid \tau].
	%	\end{align*}
		%where $E_\kappa$ denotes the event that a random coloring agrees with $\kappa$ on the neighbors $N(v)$ of $v$.
		Fix any coloring $\kappa$ of the neighbors of $v$ compatible with $\tau$, and denote by $d_i$ the number of neighbors colored with color $i$, and thus $\sum_i d_i=d$.
		We have that 
		\begin{align}
			\prob_T[\Phi(v)=1\mid \kappa] = \frac{w^{d_1}}{w^{d_1}+\sum_{i\ge 2} w^{d_i}}.
		\end{align}
		The sum consists of $q-1$ terms. Using the AM-GM inequality, we get
		\begin{align*}
			\sum_{i\ge 2} w^{d_i} &\geq (q-1) \prod_{i\ge 2} w^{d_i/(q-1)} = (q-1) w^{\frac{1}{q-1}\sum_{i\ge 2} d_i } 
   \\
  & = (q-1)w^{ (d-d_1)/(q-1)}\geq (q-1)w^{ d/(q-1)}.
		\end{align*}
		Therefore, by the law of total probability,
		\begin{align*}
			\prob_T[\Phi(v)=1\mid \tau]&=\sum_{\kappa:N(v)\to [q]}\prob_T[\Phi(v)=1\mid\kappa]\prob_T[\kappa\mid \tau]\\
			& \leq \sum_{\kappa:N(v)\to [q]}\frac{w^{d_1}}{w^{d_1}+(q-1)w^{d/(q-1)}} \prob_T[\kappa\mid \tau]\\
			& = \frac{w^{d_1}}{w^{d_1}+(q-1)w^{d/(q-1)}}\le \frac{1}{1+(q-1)w^{d/(q-1)}},
		\end{align*}
  as desired.
\end{proof}
\begin{proof}[Proof of (ii)]
	    
%Let $\hat{T}$ be the sub tree of $T$ with vertices of distance at most 2 from $v$. %\khallil{$\hat{T}$ is $T-v$, so maybe other notation. Also, I don't think this will be used,right?}\feri{ok, we are out of letters}
Assume that the neighbors $v_1,\dots,v_\ell$ of $v$ are pre-colored with $\tau$ and the remaining neighbors $v_{\ell+1},\dots,v_d$ are not. 
Let $\tau_2$ be a partial coloring of the second neighbors of $v$ that respects the precoloring, $\tau$, of $T$, and denote by $\tau*\tau_2$ the partial coloring of $T$ formed by $\tau$ and $\tau_2$.

Then
\[
\prob_{{T}}[\Phi(v)=1\mid \tau*\tau_2]=\frac{Z^1_{{T},v}(w)}{Z^1_{{T},v}(w)+\sum_{j\ge 2}Z^j_{{T},v}(w)}=\frac{\frac{Z^1_{{T},v}(w)}{Z_{T-v}(w)}}{\frac{Z^1_{{T},v}(w)}{Z_{T-v}(w)}+\sum_{j\ge 2}\frac{Z^j_{{T},v}(w)}{Z_{T-v}(w)}}.
\]
%where $\hat T$ is the tree obtained from $T$ by deleting $v$.  
Denote the connected components of $T-v$ by $T_1,\dots,T_d$. 

First, let us make some observations. 
Let $\Phi:V(T)\to[q]$ be a random coloring of $T$ compatible with $\tau*\tau_2$ with probability proportional to $w^{m(\Phi)}$. 
Denote by $\Phi_u$ the restriction of $\Phi$ to $\{u\}$ for $u\in V(T)$ and by $\Phi_{N(u)}$ the restriction of $\Phi$ to $N(u)$, the neighborhood of $u$.
% to $V(T_i)$ $i=\ell+1,\dots,d$. 
%random coloring of the vertex $v_i$ in the tree $T_i$ for obtained by restricting $\Phi$ to $T_i$.
Then for any $i\in[q]$ we have
\[
    \frac{Z^i_{T,v}(w)}{Z_{T-v}(w)}=\mathbb{E}[w^{|\Phi_{{N(v)}}^{-1}(i)|}]=w^{|\tau_{{N(v)}}^{-1}(i)|}\cdot \prod_{j=\ell+1}^d\mathbb{E}_{}[w^{|\Phi_{v_j}^{-1}(i)|}] \le 1.
\]
Also, for any $\ell+1\le j\le d$ we have
\begin{align*}
\mathbb{E}_{}[w^{|\Phi_{v_j}^{-1}(i)|}]&=\prob_{}[\Phi(v_j)\neq i]+w\prob_{}[\Phi(v_j)= i]\\
&=1-(1-w)\prob_{}[\Phi(v_j)=i].
\end{align*}
%where these probabilities should be read as conditioned on $\tau*\tau_2$
By applying Lemma~\ref{lemma: Bernoulli opt} with $x_i=\prob_{}[\Phi(v_j)=i]$, $\alpha=1-w$ and $b$ the bound on $\prob_{}[\Phi(v_j)=i]$ from (i), we obtain that
\begin{align*}
    \prod_{i=1}^{q}\mathbb{E}_{}[w^{|\Phi_{v_j}^{-1}(i)|}]&=\prod_{i=1}^q(1-(1-w)\prob_{}[\Phi_{v_j}=i])\\
    &\ge \left(1-\frac{1-w}{1+(q-1)w^{d/(q-1)}}\right)^{{1+(q-1)w^{d/(q-1)}}}\\
    &=M(w).
\end{align*}

Since the map $x\mapsto \frac{x}{x+A}$ is monotonically increasing  on $[0,1]$ if $A\ge 0$, we can bound ${\prob_T[\Phi(v)=1~|\tau*\tau_2]}$ as follows
\begin{align*}
        \prob[\Phi(v)=1|\tau*\tau_2]=&\frac{\mathbb{E}[w^{|\Phi^{-1}_{N(v)}(1)|}]}{\mathbb{E}[w^{|\Phi_{N(v)}^{-1}(1)|}]+\sum_{i\ge 2}\mathbb{E}[w^{|\Phi_{N(v)}^{-1}(i)|}]}\\
        \le&\frac{1}{1+\sum_{i\ge 2}\mathbb{E}[w^{|\Phi_{N(v)}^{-1}(i)|}]}\\
        \stackrel{AM-GM}{\le} &\frac{1}{1+(q-1)\prod_{i\ge 2}\left(\mathbb{E}[w^{|\Phi_{N(v)}^{-1}(i)|}]\right)^{1/(q-1)}}\\
        \le &\frac{1}{1+(q-1)\prod_{i=1}^q\left(\mathbb{E}[w^{|\Phi_{N(v)}^{-1}(i)|}]\right)^{1/(q-1)}}\\
        =&\frac{1}{1+(q-1)\prod_{i=1}^q\left(w^{|\tau_{N(v)}^{-1}(i)|}\cdot \prod_{j=\ell+1}^d\mathbb{E}[w^{|\Phi_{v_j}^{-1}(i)|}]\right)^{1/(q-1)}}\\
        =&\frac{1}{1+(q-1) w^{\ell/(q-1)}\left(\prod_{j=\ell+1}^d\prod_{i=1}^q\mathbb{E}[w^{|\Phi_{v_j}^{-1}(i)|}]\right)^{1/(q-1)}}\\
        \le &\frac{1}{1+(q-1)w^{\ell/(q-1)}M(w)^{(d-\ell)/(q-1)}}.
\end{align*}

To finish the proof, we use the law of total probability, similarly as in the proof of (i).
\begin{align*}
    \prob[\Phi(v)=1\mid \tau]&=\sum_{\tau_2}\prob[\Phi(v)=1~|~\tau*\tau_2]\prob[\tau*\tau_2]\\
    &\le \sum_{\tau_2}\prob[\tau*\tau_2] \frac{1}{1+(q-1)w^{\ell/(q-1)}M(w)^{(d-\ell)/(q-1)}}\\
    &=\frac{1}{1+(q-1)w^{\ell/(q-1)}M(w)^{(d-\ell)/(q-1)}},
\end{align*}
where the summation over $\tau_2$ is over all colorings of the second neighborhood of $v$ that are compatible with $\tau$. 
\end{proof}

\subsection{Further useful bounds}
In this subsection, we will assume that $w=1-\alpha\frac{q}{d+1}$ in order to obtain more amenable bounds on certain quantities that appeared in the marginal probability bounds of the previous subsection.

We will need the following lemma.
\begin{lemma}
	\label{lemma: power -1/x exp bound}
	If $x \in (0,1)$, then
	\[
	(1 -x)^{-1/x} \leq e \cdot \frac{1-x/2}{1-x}.
	\]
\end{lemma}
\begin{proof}
Since the logarithm is increasing, it suffices to prove that 
\begin{equation}\label{eq:bound after log}
    -\frac{1}{x}\log(1-x)\le 1+\log\left(\frac{1-x/2}{1-x}\right).
\end{equation}
Let us compare the coefficients of the Taylor series of the left and right hand side of Equation~\eqref{eq:bound after log}. 
For the LHS of~\eqref{eq:bound after log}, we know that
\[
-\frac{1}{x}\log(1-x)=\sum_{i\ge 0}\frac{1}{i+1}x^i=1+\sum_{i\ge 1}\frac{1}{i+1}x^i.
\]
While for the RHS of~\eqref{eq:bound after log}, we have
\begin{align*}
1+\log\left(\frac{1-x/2}{1-x}\right)&=1+\log(1-x/2)-\log(1-x)\\
    &=1-\sum_{i\ge 1}\frac{1}{i2^i}x^i+\sum_{i\ge 1}\frac{1}{i}x^i\\
    &=1+\sum_{i\ge 1}\frac{1}{i}\left(1-\frac{1}{2^i}\right)x^i.
\end{align*}
It is thus sufficient to prove that for any $i\ge 1$ we have 
\[
    \frac{1}{i+1}\le \frac{1}{i}\left(1-\frac{1}{2^i}\right), 
\]
or equivalently
\[
    i+1\le 2^i.
\]
This is true for any $i\ge 1$, since $t\mapsto 2^t$ is strictly convex, and the graphs of $t\mapsto 2^t$ and $t\mapsto t+1$ intersect at $t=0$ and $t=1$.
\end{proof}

\begin{lemma}\label{lemma:usefull bound}
Let $q\ge 3$, $d\geq q+1$ and $\alpha\in(0,1]$. 
Then for $w=1-\alpha\frac{q}{d+1}$ and $a=(d+1)/q$ we have
\begin{enumerate}[label =(\roman*)]
\item 
\[
    w^{-(d+1)/q}\le e\cdot \frac{a-1/2}{a-1},
\]
\item and
\[
    M(w)^{- \frac{d+1}{q}} \leq e \left(1+\frac{e/2}{d+1}\frac{1}{\tfrac{a-1}{a-1/2}-\tfrac{e}{d+1}}\right).
    %M(w)^{-(d+1)/q}\le e \left(1+\frac{e/2}{d+1}\cdot \frac{a-1}{a-1-\tfrac{e}{q}(1-\tfrac{1}{2a})}\right).
\]
\end{enumerate}
\end{lemma}
\begin{proof}
For the first part, we rewrite the left hand side and apply Lemma~\ref{lemma: power -1/x exp bound}, obtaining
\begin{align*}
    w^{-(d+1)/q}=\left(1-\alpha\frac{q}{d+1}\right)^{-(d+1)/q}\le \left(1-\frac{q}{d+1}\right)^{-(d+1)/q}
    \le e\cdot \frac{1-\tfrac{q}{2(d+1)}}{1-\tfrac{q}{d+1}}=e\cdot \frac{a-1/2}{a-1},
  %  &\le \frac{e}{q+2}\left(1+\frac{q}{4}\right)\\
  %  &\le \frac{7e}{20}
\end{align*}
as desired.

To prove the second part, we first bound $M(w)^{-(d+1)/q}$.
We have
\begin{align*}
        M(w)^{ -\frac{d+1}{q}}  & = \left( \left(1-\frac{1-w}{1+(q-1)w^{d/(q-1)}}\right)^{{1+(q-1)w^{d/(q-1)}}} \right)^{ -\frac{d+1}{q}}\\
  %     &= \left( \left(1- \alpha \frac{q}{d+1} \cdot\frac{1}{1+(q-1)w^{d/(q-1)}}\right)^{{1+(q-1)w^{d/(q-1)}}} \right)^{ -\frac{d+1}{q}}\\
        & \leq \left( \left(1-  \frac{q}{d+1} \cdot \frac{1}{1+(q-1)w^{d/(q-1)}}\right)^{{1+(q-1)w^{d/(q-1)}}} \right)^{ -\frac{d+1}{q}}.
\end{align*}
The last expression is again of the form $(1-x)^{-1/x}$, where $x = \frac{q/(d+1)}{1+(q-1)w^{d/(q-1)}} \in (0,1)$.
We can apply Lemma~\ref{lemma: power -1/x exp bound} to find that
\begin{align}\label{eq:bound on M(w) 1}
    M(w)^{- \frac{d+1}{q}} \leq e  \frac{1-x/2}{1-x}=e\left(1+\frac{x/2}{1-x}\right).
\end{align}
Note that by AM-GM and part (i) we have
\begin{align*}
x=\frac{q}{d+1} \cdot \frac{1}{1+(q-1)w^{d/(q-1)}}\leq \frac{q}{d+1} \cdot  \frac{1}{qw^{d/q}} \leq \frac{1}{(d+1)w^{(d+1)/q}}
\leq \frac{e}{d+1}\frac{a-1/2}{a-1}.
\end{align*}
Next, we note that $\frac{e}{d+1}\frac{a-1/2}{a-1}\leq 7 e/20<1$, which follows since $d+1\geq q+2$ and $q \ge 3$.
Using that $x\mapsto \tfrac{1-x/2}{1-x}$ is increasing for $x\leq 1$ and plugging this in into~\eqref{eq:bound on M(w) 1}
we obtain
\begin{align}\label{eq:bound on M(w) 2}
    M(w)^{- \frac{d+1}{q}} \leq e \left(1+\frac{e/2}{d+1}\frac{1}{\tfrac{a-1}{a-1/2}-\tfrac{e}{d+1}}\right),
\end{align}
as desired.
% We finally compute 
% \[
% \frac{a-1}{a-1/2}-\frac{e}{d+1}=\frac{(d+1)(d+1-q)-e(d+1-q/2)}{(d+1)(d+1-q/2)}=\frac{d+1-q-e\left(1-\tfrac{q}{2(d+1)}\right)}{d+1-q},
% \]
% implying the claimed bound.
% Note that $(x-1/2)/(x-1)$ is decreasing for $x \ge 1$. 
% We know that
% \[
%     (d+1)w^{d/q}\ge (d+1)w^{(d+1)/q}\geq \frac{d+1}{e} \cdot \dfrac{d+1 -q}{d+1 - q/2},
% \]
% by the proof of part (i).
% Since this fraction in is increasing in $d+1$ we can lowed bound this by $\tfrac{d+1}{e}\cdot \tfrac{4}{q+4}$ by plugging in $d+1=q+2$ in the fraction.
% % \[
% % (d+1)w^{d/q} \geq \frac{d+1}{e} \cdot \dfrac{d+1 -q}{d+1 - q/2}.
% % \]
% Note that this expression is increasing in $d$ for $d \geq q+4$. \khallil{proof it?}
% Hence, for $d \geq q+4$, we obtain
% \[
% (d+1)w^{d/q}  \ge \frac{q+5}{e} \frac{5}{q/2 +5}.
% \]
%
% This formula is increasing in $q$, and since we assume $q\ge 3$, we have
% \[
% (d+1)w^{d/q}  \geq \frac{8}{e} \frac{5}{3/2 +5} \ge 2.
% \]
% Therefore 
% \begin{align*}
%       M(w)^{- \frac{d+1}{q}} &\leq e \cdot 
%  \left( 
% \dfrac{(d+1)w^{ d/q} -1/2 }{(d+1)w^{ d/q} -1 } \right)\\
% & \leq e \cdot 
%  \left( 
% \dfrac{2-1/2 }{2 -1 } \right)
% \\
% & = \frac{3}{2}e.
% \end{align*}
\end{proof}
The lemma allows us to prove convenient bounds on the quantities $B(0)$ and $B(d).$ Let us recall that for $a>1$ we have
\begin{equation*}
    K(a)=e\cdot \frac{a-1/2}{a-1}.
\end{equation*}
\begin{corollary}\label{cor: bounds on B(d)}
Let $q\geq 3$ be an integer and let $d\geq q+1$. Define $a>1$ by $aq=d+1$.
\begin{enumerate}
    \item[(i)] Then $B(d)\leq K(a)/q$.
\item[(ii)] For any integers $q\geq 3$ and $d\geq q+2$ there exists a constant $K\leq \min\{ q,13\}$ such that $B(0)\leq K/q$.
Moreover, $K=e+O(1/d)$.
\end{enumerate}
\end{corollary}%\label{lem: bounds on B(d)}
\begin{proof}
  Part (i) is a direct corollary from Lemma~\ref{lemma:usefull bound}, as
  \[
    B(d)=\frac{1}{1+(q-1)w^{d/(q-1)}}\le \frac{1}{q}w^{-d/q}\le \frac{1}{q}w^{-(d+1)/q}\le \frac{K(a)}{q}.
  \]

  To prove part (ii), first observe that $B(0)<1$, thus $K\le q$. On the other hand,
   we see that for fixed $q$ we have by part (ii) of the previous lemma,
  \begin{align*}
  B(0) &= \frac{1}{1+(q-1)M(w)^{d/(q-1)}}  \le  \frac{1}{q}M(w)^{-(d+1)/q}\\
  &\leq \frac{e}{q} \left(1+\frac{e/2}{d+1}\frac{1}{\tfrac{a-1}{a-1/2}-\tfrac{e}{d+1}}\right)\\
  & =  \frac{e}{q} \left(1+\frac{e/2}{d+1} \frac{(d+1)(d+1- q/2)}{(d+1)( d+1 -q) - e(d+1-q/2)} \right)\\
  & <   \frac{e}{q} \left(1+\frac{e/2}{d+1} \frac{d+1- q/2} { d+1 -q - e} \right).
  \end{align*}
%where $K= 1 +O(1/d)$.
Notice that both $\frac{d+1- q/2} { d+1 -q - e}$ and $\tfrac{e/2}{d+1}$ are decreasing in $d+1$ and positive since $d+1\ge q+3$.
Therefore we can choose
\begin{align*}
K=e\left(1+\frac{e/2}{d+1} \frac{d+1- q/2} { d+1 -q - e}\right) \leq e\left(1+\frac{e/2}{q+3} \frac{q/2+3} {3 - e}\right)\leq e\left(1+\frac{e/2}{6} \frac{9/2} {3 - e}\right)\leq 12.6<13,
\end{align*}
and $K=e+O(1/d).$
 \end{proof}

\subsection{Bounding the local weights}
We are now in a position to start proving the promised bounds on the $\lambda_v$.
Let us briefly recall the setup and notation.
Let $q,d\in \mathbb{N}_{\geq 2}$ and $w>0$.
Let $(T,v)\in \mathcal{T}_{d+1}$ be a rooted tree and let $\Lambda\subset V(T)\setminus \{v\}$. 
Let $\tau, \tau'$ be two boundary conditions on $\Lambda$.
To increase readability, in what follows we will mostly omit the dependence on $w$ of the square root ratios and functions derived from these.
We denote the square root ratio of $v$ in $T$ for $\tau$ (resp. $\tau'$) by $X_v=X_v(w)$ (resp. $X'_v=X'_v(w)$). Recall that $Y_v(t)=tX_v+(1-t)X'_v$ and $\lambda_v=\max_{i\in[q],t\in[0,1]}\tfrac{\sqrt{S(Y_v(t))}}{S_i\left(  Y_v(t)  \right)}$,
where \[S_i(Y_v(t))=S(Y_v(t))+(w-1)Y_v(w;t)^2_i=\sum_{k=1}^q Y_v(t)_k^2+(w-1)Y_v(t)^2_i.\]
Denote by $v_1,\ldots,v_{d'}$ the neighbors of $v$ (and note that $d'\leq d$).
%Let $\freen$ be the set of free neighbors of $v$ in $T$. 
For each $j=1,\ldots,d'$, let $X_{v_j}$ (resp. $X'_{v_j}$) be the square root ratios of $v_j$ in the component of $T-v$ containing $v_j$ for the boundary condition $\tau$ restricted to this component (resp. for $\tau'$). 
For the remainder of this section let $t$ be an optimizer appearing in the definition of $\lambda_v$~\eqref{eq:def weight} and let $S(t)\coloneqq S(Y_v(t))$.
Recall that $\freen$ denotes the index set of the free neighbors of $v$.

\begin{lemma}
Assume, with the notation established above, that $\dist(\Delta_{\tau,\tau'},v)>1$. Then for any $t\in [0,1]$,
    \[
        S(t)\ge qw^{\frac{d-|\freen|}{q}}\prod_{j\in\freen}(1-(1-w)B(\ell_j))^{\frac{1}{qB(\ell_j)}},
    \]
    where $\ell_j$ is the number of fixed children of $v_j$.
    In particular,
    \[
       \frac{1}{S(t)}\le \frac{1}{q}w^{-\frac{d}{q}} \qquad\textrm{and} \qquad  \frac{1}{S(t)}\le \frac{1}{q}w^{-\frac{d-|\freen|}{q}}(1-(1-w)B(d))^{-\frac{|\freen|}{qB(d)}}.
    \]
    \label{lemma: bound on S_n}
\end{lemma}
\begin{proof}    
    We compute
    \begin{align*}
        S(t) = \sum_{i=1}^q Y_v(t)_{i}^2 \stackrel{\text{AM-GM}}{\geq} q \prod_{i=1}^q (Y_v(t)_{i}^2)^{1/q}.
    \end{align*}
    %Recall that $Y_v(t) = tX_v + (1-t)X'_v$.
    We apply the weighted AM-GM inequality and obtain the lower bound
    \begin{equation}
        q \prod_{i=1}^q (Y_v(t)_{i}^2)^{1/q}  \geq  q \prod_{i=1}^q  \left( X_{v,i}\right)^{2t/q}  \left( X'_{v,i}\right)^{2(1-t)/q} .
    \end{equation}
    The RHS can be expressed in terms of the square root ratios of the neighbors $v_1, \ldots, v_{d'}$ of $v$:
    \begin{align*}
        q \prod_{i=1}^q  \left( X_{v,i}\right)^{ \frac{2t}{q} }  \left( X'_{v,i}\right)^{\frac{2(1-t)}{q}}  &=  q \prod_{i=1}^q  \left( \prod_{j=1}^{d'} F(X_{{v_j} })_i^2 \right)^{ \frac{t}{q} }    \left(  \prod_{j=1}^{d'} F(X'_{{v_j} })_i^2 \right)^{\frac{(1-t)}{q}} \\
        &=  q \left(\prod_{i=1}^q   \prod_{j=1}^{d'} F(X_{{v_j} })_i^2 \right)^{ \frac{t}{q} }    \left(\prod_{i=1}^q  \prod_{j=1}^{d'} F(X'_{{v_j} })_i^2 \right)^{\frac{(1-t)}{q}} \\
        &=  q \left(\prod_{j=1}^{d'} \prod_{i=1}^q    F(X_{{v_j} })_i^2 \right)^{ \frac{t}{q} }    \left(\prod_{j=1}^{d'} \prod_{i=1}^q   F(X'_{{v_j} })_i^2 \right)^{\frac{(1-t)}{q}} \\
        &=  q \left(w^{d'-|\freen|}\prod_{j\in\freen}\prod_{i=1}^q    F(X_{{v_j} })_i^2 \right)^{ \frac{t}{q} }    \left(w^{d'-|\freen|}\prod_{j\in\freen} \prod_{i=1}^q   F(X'_{{v_j} })_i^2 \right)^{\frac{(1-t)}{q}} \\
        &\geq   q w^{\frac{d-|\freen|}{q}}\prod_{j\in\freen}\left(\prod_{i=1}^q    F(X_{{v_j} })_i^2 \right)^{ \frac{t}{q} }    \left( \prod_{i=1}^q   F(X'_{{v_j} })_i^2 \right)^{\frac{(1-t)}{q}}.
        % &=  q \prod_{i=1}^q  \left( \prod_{k=1}^{2d} F(\widehat{X}_{k})_i^2 \right)^{ \lambda_k },
        % & = q \exp( \log(  \prod_{k=1}^{2d}  \prod_{i=1}^q   F(\widehat{X}_{k})_i^{ 2\lambda_k }) ) \\
        % & = q \exp(   \sum_{k=1}^{2d}  \sum_{i=1}^q \lambda_k \log(  F(\widehat{X}_{k})_i^{ 2 }) ), 
    \end{align*}
%think about F of ratio instead of just the ratio\

    Note that $X^2_{v_j,i} / S(X_{v_j})$  and $\left(X'\right)^2_{v_j,i} / S(X'_{v_j})$ are the probabilities that vertex $v_j$ is colored $i$, given boundary conditions $\tau$ and $\tau'$ respectively.
    From the definition of $F$ it then follows that 
    \begin{align*}
    F(X_{v_j})_i^2 &= 1 - (1-w) \prob_{T-v;w}[\Phi(v_j)=i ~|~\tau], \text{ and}
    \\
    F(X'_{v_j})_i^2 &= 1 - (1-w) \prob_{T-v;w}[\Phi(v_j)=i ~|~\tau']  .  
    \end{align*}
Regardless of the boundary condition $\eta \in \{ \tau, \tau'\}$, we may apply Lemma~\ref{lem:prob basic} and Lemma~\ref{lemma: Bernoulli opt} to see that
\begin{align*}
    \prod_{i=1}^q\left(1-(1-w)\prob_{T-v;w}[\Phi(v_j)=i ~|~\eta]\right)\ge (1-(1-w)B(\ell_j))^{\frac{1}{B(\ell_j)}},
\end{align*}
where $\ell_j$ denotes the number of fixed children of $v_j$.
Observe that this quantity only depends on $v_j$ and not on the choice of $X_{v_j}$ or $X'_{v_j}$, since for both boundary conditions $\tau$ and $\tau'$, the vertex $v_j$ has exactly the same fixed neighbors, since $\dist(v,\Delta_{\tau,\tau'})>1$.%\khallil{I think it only matters that they have the same number of fixed nbrs, even if the colors would differ.}\feri{}
  Thus we obtain that
    \[
    S(t)\ge qw^{\frac{d-|\freen|}{q}}\prod_{j\in\freen}(1-(1-w)B(\ell_j))^{\frac{1}{qB(\ell_j)}},
    \]
as desired.

To obtain the other two statements, we have to use the monotonicity established before Lemma~\ref{lem:prob basic}.
\end{proof}

We can now finally provide a bound on $\lambda_v^2(w)$ and prove Proposition~\ref{prop: bound lambda_v2}, which we restate here for convenience.
\boundlambdav*

\begin{proof}
We will write $S(t)\coloneqq S(Y_v(w;t))$ and $\lambda_v=\lambda_v(w)$ throughout the proof.
By definition we can write 
\[
\lambda_v^2=\frac{1}{S(t)(1-\frac{1-w}{S(t)}Y_v(w;t)_i^2)^2},
\]
for some $i\in [q]$ and $t\in[0,1]$, which can then be bounded as
\begin{equation}\label{eq:double S bound}
    \lambda_v^2\le \frac{1}{S(t)(1-\frac{1-w}{S(t)})^2},
\end{equation}
since $Y_v(w;t)_i\le 1$ for any $i\in [q]$ and $t\in[0,1]$.

% Since the  function $x\mapsto (1-y/x)^{-2}$ is a decreasing function on $(y,\infty)$ if $y>0$, and since $Y_v(t)_i \leq 1$ for all $i \in [q]$ and $t \in [0,1]$ ,
% %\khallil{maybe reword this and double check}, it follows by definition of $\lambda_v$  
% we have for any ($t$-independent) bounds ${S}_1$ and ${S}_2$ on $S$ respectively, that
% \begin{equation}\label{eq:double S bound}
%     \lambda_v^2 \le \frac{1}{{S_1}} \dfrac{1}{\left(1- \frac{(1-w)}{{S_2}}\right)^2}.
% \end{equation}
In what follows we write $S$ instead of $S(t).$
We now start with the proof of (i). 
To further bound it $\lambda_v$ we will use the previous lemma to bound $1/S$ and $(1-\frac{1-w}{S})^{-2}$ separately.

Let us apply Lemma~\ref{lemma: bound on S_n} with $\freen=\emptyset$, noting the expression is decreasing in $|\freen|$, to obtain the bound
% For $S_2$, we may use the second bound on $S$ from Lemma~\ref{lemma: bound on S_n} with $\freen=\emptyset$, noting the expression is decreasing in $|\freen|$, to obtain the bound 
    \[
         \lambda_v^2 \le \frac{1}{{S}}\left(1-\frac{1-w}{qw^{d/q}}\right)^{-2}\leq \frac{1}{{S}} \left(1-\frac{1-w}{qw^{(d+1)/q}}\right)^{-2},
    \]
    % \[
    %     (1-(1-w)/S)^{-2}\le \left(1-\frac{1-w}{qw^{d/q}}\right)^{-2}\leq \left(1-\frac{1-w}{qw^{(d+1)/q}}\right)^{-2},
    % \]
which can be further bounded using Lemma~\ref{lemma:usefull bound} by
    \[
   \frac{1}{S} \left(1-\frac{e}{(d+1)}\frac{a-1/2}{a-1}\right)^{-2}= \frac{1}{S}\left(1-\frac{K(a)}{(d+1)}\right)^{-2}.
    \]
To bound $1/S$, we apply the second part of Lemma~\ref{lemma: bound on S_n} again, without the assumption $\freen = \emptyset$. 
%\khallil{It's not using $\ell=d$, but the second part of the lemma, since $\ell$ denotes the fixed neighbours of $v$}
This yields
\begin{align*}
\frac{1}{S}&\leq \frac{1}{q}\left(w^{\tfrac{-d}{q}}\right)^{\tfrac{d-f}{d}}\cdot \left(\left(1-(1-w)B(d)\right)^{\tfrac{-d}{qB(d)}}\right)^{\tfrac{f}{d}}
\\
&\leq \frac{1}{q}\left(w^{\tfrac{-(d+1)}{q}}\right)^{\tfrac{d-f}{d}}\cdot \left(\left(1-(1-w)B(d)\right)^{\tfrac{-(d+1)}{qB(d)}}\right)^{\tfrac{f}{d}},
\end{align*}
which by Lemma~\ref{lemma: power -1/x exp bound} and Lemma~\ref{lemma:usefull bound} can be further bounded by
\[
\frac{1}{q}\left(e\cdot \frac{a-1/2}{a-1}\right)^{\tfrac{d-f}{d}}\cdot \left(e\cdot \frac{a-B(d)/2}{a-B(d)}\right)^{\tfrac{f}{d}}.
\]
Since $x\mapsto \tfrac{a-x/2}{a-x}$ is increasing, we can plug in the upper bound on $B(d)$ from Corollary~\ref{cor: bounds on B(d)} and obtain the desired expression.
%Combining the two bounds now completes the proof of part (i).

To prove part (ii), we use Lemma~\ref{lemma: bound on S_n} with $\freen=\{1,\ldots, d\}$, $\ell_j=0$ for every $j\in\freen$ and Lemma~\ref{lemma: power -1/x exp bound}.
This gives us, 
\begin{align}\label{eq:first bound on S ii}
\frac{1}{S}\leq \frac{1}{q}\left(\left(1-(1-w)B(0)\right)^{\tfrac{-(d+1)}{qB(0)}}\right)\leq\frac{e}{q}\frac{a-B(0)/2}{a-B(0)}.
\end{align}
Similarly,
\begin{align}
\frac{1-w}{S}\leq \frac{e}{d+1} \cdot \frac{a-B(0)/2}{a-B(0)}.\label{eq:second bound on S ii}
\end{align}

Using that $B(0)\leq K/q$ for a constant $K\le \min\{ q,13\}$, by Corollary~\ref{cor: bounds on B(d)} and recalling that $a=(d+1)/q$, we obtain, 
\begin{equation}\label{eq:somewhat precis bound on 1/S}
\frac{a-B(0)/2}{a-B(0)}\leq \frac{d+1-K/2}{d+1-K}=\left(1+\frac{K/2}{d+1-K}\right).
\end{equation}
Using that $K\leq \min\{ q,13\}$, $d+1\geq 6$ and $d\geq q+2$ we get $e\left(1+\tfrac{K/2}{d+1-K}\right)\le\min\{9,q+2\}$ and as $K=e+O(1/d)$, therefore $e\left(1+\frac{K/2}{d+1-K}\right)=e+O(1/d)$. Now by using $K'=\max\left(K,e\left(1+\frac{K/2}{d+1-k}\right)\right)\le \max\{{q+2,13}\}$ we obtain the bound on \eqref{eq:first bound on S ii} and \eqref{eq:second bound on S ii} as
\[
\frac{1}{S}\le \frac{e}{q}\left(1+\frac{K'/2}{d+1-K'}\right) \qquad\textrm{and}\qquad  \left(1-\frac{1-w}{S}\right)^{-2}\le \left(1-\frac{K'}{d+1}\right)^{-2},
\]
thus by substituting these bound into \eqref{eq:double S bound}, we obtain the desired bound. Moreover, by construction we have $K'=e+O(1/d)$.
% To prove part (iii), we notice that $S_i(Y_v(t)) \leq S(Y_v(t))$ for all $i\in [q]$ and $t \in [0,1]$. 
% Therefore 
% \[
% \lambda_v^2 \ge \max_{t\in[0,1]}\dfrac{1}{S(Y_v(t))}. 
% \]
% Moreover,  $S(Y_v(t)) = \sum_{i \in [q]} Y_v(t)_i^2$, with each $Y_v(t)_i$ being a convex combination of $X_{v,i}$ and $Y_{v,i}$. 
% The latter two quantities are components of (square root) ratio vectors and are always bounded from above by $1$. 
% It follows that $S(Y_v(t)) \leq q$, and hence, 
% \[
% \lambda_v^2  \ge \frac{1}{q}.
% \]
\end{proof}
\section{Concluding remarks and questions}\label{sec:conclusion}
In this paper we managed to prove weak and strong spatial mixing for the $q$-state Potts model on bounded degree trees for a near optimal range of parameters of the form $w\geq 1-\alpha \tfrac{q}{d+1}$ with $\alpha=1-K/d$, where $d+1$ is the degree bound and $K$ a constant.
Unfortunately, our approach requires $d+1\geq \tfrac{e-1/2}{e-1}q$ for strong spatial mixing. 
It would of course be nice to get rid of this constraint.
Possibly, the local modification of the $2$-norms from~\cite{chen2023strong} could be helpful here, but it is unclear to us how to utilize these in our setting.

We have not attempted to obtain concrete values on the constants $K, K''_q$ for small values of $d$. 
There are several places where one could improve our results for small values of $d$. In particular in the proofs of our main theorems where one could replace the contraction rate of $\tfrac{d}{d+1}$ by something much closer to $1$, in Lemma~\ref{lem: beta bound} where one could directly work with $(\tfrac{d}{d+1})^{d+1}$ rather than bounding it by $e^{-1}$, in Lemma~\ref{lem:prob basic} where it would be better to work with concrete bounds for small values of $d$, and in Lemma~\ref{lemma:usefull bound} and Corollary~\ref{cor: bounds on B(d)} where we simply replaced $d$ by $d+1$ in the respective proofs. 

%It would also be nice to improve our bounds on the the constant $K$ so as to make our results more meaningful for small values of $d$. 

Our approach for bounded degree trees can be extended to show the partition function of the tree does not vanish on a complex neighbourhood of the interval for which the model has SSM following the approach from~\cite{LSScorrelation}.   
It would be interesting to conclude the same for large girth graphs. 
More precisely, let $d,q\geq 3$ be integers and assume (for simplicity) that $d+1\geq \tfrac{e-1/2}{e-1}q$ and write $\alpha=1-K/d$ where $K$ is the constant from Theorem~\ref{thm:main ssm}. Is it true that there exists a constant $g$ and an open set $U\subset \mathbb{C}$ that contains the interval $(1-\alpha\tfrac{q}{d+1},1]$ such that for any graph $G$ of girth at least $g$ and maximum degree at most $d+1$ its partition function $Z_G$ (defined as in~\eqref{eq:def pf}) satisfies $Z_G(w')\neq 0$ for all $w'\in U$?
If this were true, then this would yield a \emph{deterministic} polynomial time approximation algorithm for computing $Z_G$ via Barvinok's interpolation method~\cite{Bar,PR17}.

\section*{Acknowledgments}
We thank the anonymous referees for their constructive feedback, and for correcting an error in an earlier version.

\bibliographystyle{alphaurl}
\bibliography{main_bib}

\end{document}